\documentclass[a4paper,10pt]{scrartcl}
\usepackage[authoryear,compress]{natbib}
\usepackage[latin1]{inputenc}
\usepackage{amsmath,amsthm}
\usepackage{a4wide}
\usepackage{amsfonts}
\usepackage{amssymb}
\usepackage{graphicx}
\usepackage{subfigure}
\usepackage{pgfplots}
\usepackage{tikz}
\usepackage{mathtools}
\usepackage{booktabs}
\usepackage{enumitem}
\setenumerate{label=\emph{(\roman*)}}

\pgfplotsset{compat=newest}
\newcommand{\X}{\mathcal X}
\newcommand{\Y}{\mathcal Y}
\newcommand{\Rset}{\mathbb R}
\newcommand{\ind}{\mathbf 1}

\newcommand{\gdag}{Tf}
\newcommand{\Tr}{\mathrm{Trace}}
\newcommand{\E}[1]{\mathbb E \left[#1\right]}
\newcommand{\sE}[1]{\mathbb E [#1]}
\newcommand{\Var}[1]{\mathrm{Var} \left[#1\right]}
\newcommand{\Prob}[1]{{\mathbb P}\left\{ #1 \right\}}
\DeclareMathOperator*{\argmin}{argmin}

\newlength\fheight \newlength\fwidth

\newcommand{\normX}[1]{\left\|{#1}\right\|_{\X}}
\newcommand{\normY}[1]{\left\|{#1}\right\|_{\Y}}
\newcommand{\snormX}[1]{\|{#1}\|_{\X}}
\newcommand{\snormY}[1]{\|{#1}\|_{\Y}}
\newcommand{\abs}[1]{\left|{#1}\right|}
\newcommand{\inner}[2]{\langle#1,#2\rangle}

\definecolor{mycolor1}{rgb}{0.00000,0.75000,0.75000}

\begin{document}
\newtheorem{theorem}{Theorem}[section]
\newtheorem{proposition}[theorem]{Proposition}
\newtheorem{lemma}[theorem]{Lemma}
\newtheorem{corollary}[theorem]{Corollary}

\newtheorem{definition}{Definition}
\newtheorem{assumption}{Assumption}
\newtheorem{remark}{Remark}
\newtheorem{example}{Example}

\begin{center}
\begin{minipage}{.8\textwidth}
\centering 
\LARGE Empirical Risk Minimization as Parameter Choice Rule for General Linear Regularization Methods\\[0.5cm]

\normalsize
\textsc{Housen Li}\\[0.1cm]
\verb+housen.li@mathematik.uni-goettingen.de+\\
University of Goettingen, Germany.\\[0.1cm]

\textsc{Frank Werner}\footnotemark[1]\\[0.1cm]
\verb+frank.werner@mpibpc.mpg.de+\\
Max Planck Institute for Biophysical Chemistry, Goettingen, Germany\\
and\\
Felix Bernstein Institute for Mathematical Statistics in the Bioscience, University of Goettingen, Germany
\end{minipage}
\end{center}

\footnotetext[1]{Corresponding author}

\begin{abstract}
We consider the statistical inverse problem to recover $f$ from noisy measurements $Y = Tf + \sigma \xi$ where $\xi$ is Gaussian white noise and $T$ a compact operator between Hilbert spaces. Considering general reconstruction methods of the form $\hat f_\alpha = q_\alpha \left(T^*T\right)T^*Y$ with an ordered filter $q_\alpha$, we investigate the choice of the regularization parameter $\alpha$ by minimizing an unbiased estimate of the predictive risk $\sE{\Vert Tf - T\hat f_\alpha\Vert^2}$. The corresponding parameter $\alpha_{\mathrm{pred}}$ and its usage are well-known in the literature, but oracle inequalities and optimality results in this general setting are unknown.
We prove a (generalized) oracle inequality, which relates the direct risk $\sE{\Vert f - \hat f_{\alpha_{\mathrm{pred}}}\Vert^2}$ with the oracle prediction risk $\inf_{\alpha>0}\sE{\Vert Tf - T\hat f_{\alpha}\Vert^2}$. From this oracle inequality we are then able to conclude that the investigated parameter choice rule is of optimal order in the minimax sense.

Finally we also present numerical simulations, which support the order optimality of the method and the quality of the parameter choice in finite sample situations.
\end{abstract}


\textit{AMS classification numbers:} Primary 62G05; Secondary 62G20, 65J22, 65J20. \\[0.3cm]

\section{Introduction}
Suppose we want to recover an unknown function $f \in \X$ from noisy measurements
\begin{equation}\label{eq:model}
Y = Tf + \sigma \xi
\end{equation}
where $T : \X \to \Y$ is an operator between Hilbert spaces $\X$ and $\Y$, $\xi$ is a standard Gaussian white noise process and $\sigma>0$ denotes the noise level. In fact, model~\eqref{eq:model} has to be understood in a weak sense as $\xi \notin \Y$, i.e. for each $y \in \Y$ we have access to observations of the form
\[
Y_y \coloneqq \left\langle Tf, y\right\rangle_{\Y} + \sigma \left\langle \xi, y\right\rangle_{\Y^* \times \Y}
\]
where $\left\langle \xi, y\right\rangle_{\Y^* \times \Y} \sim \mathcal N \bigl(0, \snormY{y}^2\bigr)$ and $\sE{\left\langle \xi, y_1\right\rangle_{\Y^* \times \Y}\left\langle \xi, y_2\right\rangle_{\Y^* \times \Y}} = \left\langle y_1, y_2\right\rangle_{\Y}$ for all $y_1, y_2 \in \Y$. Models of the form \eqref{eq:model} underly a plenitude of applications, see e.g. \citet{os86}, and have been considered by \citet{mp01}, \citet{bhmr07}, \citet{iss12}, \citet{ilm14} and \citet{w15}. 

Throughout the paper we will assume that the operator $T$ is injective, compact and Hilbert-Schmidt, i.e.~the squares of its singular values are summable. Especially, this implies that its singular values tend to $0$ and hence the inversion of $T$ is unstable, i.e.~the problem to recover $f$ from \eqref{eq:model} is ill-posed and regularization is needed, see \citet{c11}, \citet{ehn96} and the references therein. In the literature many different approaches for the estimation of $f$ can be found, including methods based on a singular value decomposition (SVD) of $T$ \citep[see e.g.][]{js91,mr96,jkpr04,cg06}, wavelet-vaguelette \citep{d95} and vaguelette-wavelet methods \citep{as98}, and Galerkin-type methods \citep{chr04}. 

In this paper we follow a common paradigm and consider regularization methods written in terms of an \emph{ordered filter} $q_{\alpha}: \left[0, \left\Vert T^*T\right\Vert \right] \to \Rset$ parametrized by $\alpha>0$ (see e.g. Definition \ref{def:order_filter} for the precise requirements on $q_\alpha$), meaning that the regularized solution is given by
\begin{equation}\label{eq:estimator}
\hat f_\alpha \coloneqq q_\alpha\left(T^*T\right)T^*Y.
\end{equation}
Regularization methods of the form \eqref{eq:estimator} include famous examples like spectral cut-off and Tikhonov regularization, 
and have been studied extensively in the literature, see \citet{ehn96} and the references therein, and \citet{bhmr07}. 

Choosing an appropriate parameter $\alpha$ in \eqref{eq:estimator} is an important problem in regularization theory as it dramatically influences the performance of the estimator $\hat f_\alpha$. \citet{bhmr07} show that estimators of the form \eqref{eq:estimator} are order-optimal over certain smoothness classes $\mathcal W \subset \X$, if the parameter $\alpha = \alpha_*$ is chosen in a reasonable a-priori way (depending on properties of $\mathcal W$). This means that $\hat f_{\alpha_*}$ achieves the best possible rate of convergence w.r.t. the \textit{direct risk} $R \left(\alpha, f\right) \coloneqq \sE{\snormX{ \hat f_\alpha - f }^2}$ in these classes. In practice, the parameter $\alpha$ has to be chosen without any knowledge of $f$ (and hence of $\mathcal W$), which makes a-priori parameter choice rules useless. Therefore, \textit{a-posteriori} parameter choice rules are of interest, as they make only use of the data $Y$ and the noise level $\sigma>0$. As a-posteriori parameter choice rules have to adapt to the unknown properties of $\mathcal W$ (and hence of $f$), this issue is also known as \textit{adaptivity}. For simplicity, we will assume here that $\sigma$ is known. In practice, the parameter $\sigma$ can typically be estimated sufficiently fast from the data or the measurement setting~\citep[see e.g.][]{r84,hkt90,dmw98}. We will discuss this situation in more detail in Section~\ref{s:con_out}. 

A variety of a-posteriori parameter choice rules have been proposed in the literature, including the discrepancy principle \citep{m66,da86,l95,bhr16}, generalized cross-validation \citep{w77, ghw79,l93}, the Lepski{\u\i}-type balancing principle \citep{l90,m06,mp06,wh12} and many more. We refer to \citet{bl11} for a recent overview and numerical comparison. General adaptivity in statistical inverse problems has also been treated in \citet{g99}, \citet{t00} and \citet{cglt03}. 

In this paper we deal with a specific method based on empirical risk minimization originally introduced by \citet{m73} for model selection in linear regression and therefore known as \textit{Mallow's $C_L$}. Consider the \textit{prediction risk} $r \left(\alpha, f\right) \coloneqq \sE{\Vert T (\hat f_\alpha - f)\Vert_{\Y}^2}$. Following \citet{s81}, we find that an (up to a constant independent of $\alpha$) unbiased estimator for this quantity is given by
\begin{equation}\label{eq:risk_estimator}
\hat{r}(\alpha, Y) \coloneqq \Vert T\hat{f}_{\alpha} \Vert_{\Y}^2 - 2 \inner{Y}{T \hat{f}_{\alpha}}_{\Y^* \times \Y} + 2\sigma^2 \Tr  \left( T^*T q_\alpha \left(T^*T\right) \right).
\end{equation}
Now the idea is to choose $\alpha$ as a minimizer of $\hat r \left(\alpha,Y\right)$, i.e.
\[
\alpha_{\mathrm{pred}}  \in \argmin_{\alpha>0} \hat r \left(\alpha, Y\right).
\]
Note that the functional \eqref{eq:risk_estimator} penalizes the misfit between the model $\hat f_\alpha$ and the data $Y$, and furthermore $2\sigma^2 \Tr  \left( T^*T q_\alpha \left(T^*T\right) \right)$ penalizes the number of degrees of freedom of the model. We refer to \citet{bm01,bm07} for details and a discussion of other possible penalty terms in \eqref{eq:risk_estimator}. 

It is known that choosing $\alpha = \alpha_{\mathrm{pred}}$ in combination with certain regularization schemes leads to an order optimal method w.r.t. the prediction risk $r\left(\alpha, f\right)$, see e.g. \citet{l87}, \citet{v86} and \citet{l93}. A very precise result about its performance,  which is also a central ingredient of this paper, can be found in the seminal paper by \citet{k94}, who proves exponential deviation bounds for $r\left(\alpha_{\mathrm{pred}},f\right)$. 

Due to ill-posedness, optimality w.r.t. the prediction risk is however a very weak statement, and consequently we are interested in order optimality w.r.t. the direct risk. This question has hardly been touched in {the} literature, and to the authors' best knowledge the only result is due to \citet{cg14} who restrict to finite dimensional spaces and spectral cut-off regularization. Nevertheless, the choice $\alpha = \alpha_{\mathrm{pred}}$ has successfully been applied in image denoising applications \citep[see e.g.][]{lbu07,cslt13,wm13,dvfp14}. Moreover, the distributional behavior of $\alpha_{\mathrm{pred}}$ has recently been studied by \citet{lpbbbdw17}. There it has been argued that the choice $\alpha_{\mathrm{pred}}$ and choices based on unbiased risk minimization in general do not seem suitable for inverse problems. Besides, it remains one of the most popular parameter selection rules, due to its favorable practical performance \citep[cf.][]{bl11,cg14}. In this spirit, we will prove an \textit{oracle inequality} of the form 
\begin{equation}\label{eq:goie}
R\left(\bar \alpha,f\right) \leq \Theta \left( \min_{\alpha>0} r\left(\alpha, f\right)\right)
\end{equation}
for all $f \in \mathcal W$ as $\sigma \searrow 0$ where $\Theta$ is some (explicit) functional and $\mathcal W \subset \X$ some smoothness class. More specifically, the functional $\Theta$ will be such that we can derive order optimality of $\hat f_{\alpha_{\mathrm{pred}}}$ under weak assumptions over many classes $\mathcal W \subset \X$. Moreover, we provide a general strategy to prove estimates of the form \eqref{eq:goie} which is of independent interest and might be used for the analysis of other a-posteriori parameter choice rules as well. This makes our analytical methodology substantially different from that in \citet{cg14}, since the crucial techniques (such as optional stopping of martingales) there do not apply to general regularization schemes. For more details on oracle inequalities in statistical inverse problems we refer to \citet{cgpt02} and \citet{bhr16}.

The rest of this paper is organized as follows. In the next section we introduce ordered filter based regularization methods and the empirical parameter choice rule via unbiased prediction risk minimization. The convergence analysis of such a rule is established by means of an oracle inequality in Section~\ref{s:theory} with corresponding conclusions on order optimality. In Section~\ref{sec:proofs} we present a general methodology for proving oracle inequalities of the form \eqref{eq:goie} and apply this methodology for proving the results from Section~\ref{s:theory}. The performance, as well as the convergence behavior, of the analyzed parameter choice rule is examined by comparison with other empirical parameter choice rules in a simulation study in Section~\ref{sec:numerics}. We end this paper with some conclusions in Section~\ref{s:con_out}. {Some technical details are deferred to the Appendix.}

\section{Filter based regularization and empirical risk minimization}\label{sec:filter_minRisk}

\subsection{Ordered filter based regularization methods}\label{ss:order_filter}

Suppose that $\X$ and $\Y$ are Hilbert spaces, $T : \X \to \Y$ is an injective and compact Hilbert-Schmidt operator, and $\xi$ in \eqref{eq:model} is a standard Gaussian white noise process as described in the Introduction. To simplify the notation we will always assume that $\dim\left(\X\right) = \infty$, but stress that the case of finite dimensional $\X$ (and $\Y$) can be treated similarly. By assumption, there exists a singular value decomposition (SVD) $\left\{ \left(\sqrt{\lambda_k},e_k, g_k\right)\right\}_{k \in \mathbb N}$ of $T$ where $\lambda_1 \ge \lambda_2 \ge  \cdots > 0$ are the eigenvalues of $T^*T$, $e_1, e_2, \ldots$ are the corresponding normalized eigenvectors, and $g_k = \lambda_k^{-1/2}T e_k$ for $k = 1,2,\ldots$. By introducing the notation $Y_k \coloneqq \inner{g_k}{Y}$, $\xi_k \coloneqq \inner{g_k}{\xi}$ and $f_k = \inner{f}{e_k}$, we equivalently transform the model~\eqref{eq:model} to the Gaussian sequence model
\begin{equation}\label{eq:seq_model}
Y_k = \sqrt{\lambda_k}f_k + \sigma \xi_k, \qquad k = 1, 2, \ldots, {\qquad \text{with}\qquad \xi_k \stackrel{\text{i.i.d.}}{\sim} \mathcal N \left(0,1\right)}.
\end{equation}
As mentioned in the Introduction, we focus on regularization methods of the form \eqref{eq:estimator}, which by means of \eqref{eq:seq_model} can be equivalently formulated as 
\begin{equation}\label{eq:seq_est}
\left(\hat{f}_{\alpha}\right)_k = \sqrt{\lambda_k}q_{\alpha}(\lambda_k)Y_k, \qquad k \in \mathbb N. 
\end{equation}
\begin{remark}
Note that $T$ being Hilbert-Schmidt implies that $\sum_{k=1}^\infty \lambda_k < \infty$, and hence $T^*Y$ can be interpreted as a random variable with values in $\X$ as
\[
\E{\normX{T^*\xi}^2} = \E{\sum_{k=1}^\infty \left\langle \xi, T e_k\right\rangle_{\Y}^2} = \sum_{k=1}^\infty \lambda_k \E{\inner{\xi}{g_k}_{\Y}^2}< \infty.
\]
Consequently, $\hat f_\alpha$ as in \eqref{eq:estimator} is well-defined. 
\end{remark}
Estimators of the form~\eqref{eq:estimator} or~\eqref{eq:seq_est} can be understood as stable approximations of the well-known least squares estimate $\hat{f}\coloneqq(T^*T)^{-1}T^*Y$ (or $\hat f_k = Y_k/\sqrt{\lambda_k}$) in the sense of replacing $(\cdot)^{-1}$ with a function $q_{\alpha}(\cdot)$. To obtain a well-defined and reasonable regularization method, the functions $q_{\alpha}(\cdot)$ should satisfy proper conditions. We are particularly interested in case that $q_{\alpha}(\cdot)$ is an ordered filter. 

\begin{definition}\label{def:order_filter}
Let $q_{\alpha}: [0, \lambda_1] \to \Rset$, indexed by $\alpha \in \mathcal{A} \subset \Rset_+$, be a sequence of functions. We always assume that $\mathcal A$ is bounded {and closed, equipped with the subspace topology inherited from $\Rset_+$, and that $0 \in \mathcal A$.}
\begin{enumerate}
\item
The family $q_\alpha, \alpha \in \mathcal A$, is called a \textit{filter}, if there exist constants $C_q', C_q'' >0$ such that for every $\alpha \in \mathcal{A}$ and every $\lambda \in  [0, \lambda_1] $ 
\[
\alpha \abs{q_{\alpha}(\lambda)} \le C_q' \qquad \text{ and } \qquad \lambda \abs{q_{\alpha}(\lambda)} \le C_q''.
\] 
\item
The filter $q_\alpha, \alpha \in \mathcal A$, is called \textit{ordered}, if further the sequence $\{q_{\alpha}(\lambda_k)\}_{k=1}^{\infty}$ is strictly monotone, i.e.
\[
{\alpha_1 > \alpha_2} \qquad\Rightarrow \qquad \forall~k\in \mathbb N: q_{\alpha_1}(\lambda_k) \le q_{\alpha_2}(\lambda_k)\quad\text{and}\quad\exists~k_0 \in \mathbb N: q_{\alpha_1}(\lambda_{k_0}) < q_{\alpha_2}(\lambda_{k_0}),
\]
and continuous as $\mathcal A \ni \alpha \mapsto \{q_{\alpha}(\lambda_k)\}_{k=1}^{\infty} \in \ell^2$.
\end{enumerate}
\end{definition}
The requirement of an ordered filter is rather weak, as it is satisfied by various regularization methods. In Table~\ref{tab:example} we give several examples of such. Note that for spectral cut-off regularization, the set $\mathcal{A}$ has to be chosen as $\{\lambda_k: k = 1, 2, \ldots\}\cup\{0\}$ in order to guarantee the strict monotonicity and the continuity required by condition (ii) in Definition~\ref{def:order_filter}. In Table~\ref{tab:example} we also indicate whether the method can be implemented without SVD. This property is extremely crucial in practice, especially for large-scale applications, where the computation of an SVD is often impossible given limited time and resources. The implementation of Showalter's method, for instance, can avoid SVD by employing Runge-Kutta schemes, see e.g.~\citet{r05}. For a further discussion of these and other methods we refer to the monograph by~\citet{ehn96}.

In this paper, we focus on the asymptotic properties of ordered filter based regularization methods as the noise level $\sigma$ goes to zero. As noticed by \citet{b84}, the convergence rate of {any} regularization method can be arbitrarily slow if the underlying problem {is} ill-posed. In order to derive convergence rates we need {to} assume some smoothness about the unknown truth $f$. Typically, the smoothness of $f$ is measured relative to the smoothing properties of the forward operator $T$ in terms of a \emph{source condition}, i.e.~we assume that
\begin{equation}\label{eq:sc_general}
f \in \mathcal{W}_{\phi}{(\rho)} \coloneqq \left\{f \in \X : f = \phi(T^*T)w, \normX{w} \le {\rho} \right\}\qquad \text{ for some constant } {\rho}, 
\end{equation}
where {$\phi: \Rset_+ \to \Rset_+$} is a so-called \emph{index function}, i.e.~$\phi$ is continuous, strictly increasing, and $\phi(0) = 0$. For any $f \in \X$ there exist a function $\widetilde \phi$ and a constant $\tilde\rho$ such that $f \in \mathcal W_{\widetilde \phi}(\tilde\rho)$, cf.~\citet{mh08}.

To take advantage of~\eqref{eq:sc_general} we furthermore assume that $\phi$ is a \emph{qualification} of the filter $q_\alpha$, this is
\begin{equation}\label{eq:qc_general}
\sup_{\lambda \in [0, \lambda_1]} \phi(\lambda) \abs{1 - \lambda q_{\alpha}(\lambda)} \le C_{\phi} \phi(\alpha)  \qquad \text{ for all } \alpha \in \mathcal{A},
\end{equation}  
with $C_{\phi}$ being a constant depending only on $\phi$. For further details on general source conditions and corresponding qualifications we refer to \citet{mp03}. 
As an example consider $\phi(t) = t^{v}$, which is known as H{\"o}lder type source condition of order $v>0$: 
\begin{equation}\label{eq:sc_hoelder}
\mathcal{W}_{v}(\rho) \coloneqq \left\{f \in \X : f = (T^*T)^v w, \normX{w} \le \rho \right\}.
\end{equation}
The function $\phi(t) = t^{v}$ is a qualification of the filter $q_\alpha$ if
\begin{equation}\label{eq:qc_hoelder}
\sup_{\lambda \in [0, \lambda_1]}  \lambda^v \abs{1 - \lambda q_{\alpha}(\lambda)} \le C_v \alpha^v \qquad \text{ for all } \alpha \in \mathcal{A}.
\end{equation}
In this case, the largest possible $v$ such that \eqref{eq:qc_hoelder} is satisfied, is called the classical or polynomial \emph{qualification index} $v_0$ of the ordered filter $q_{\alpha}$. For the methods discussed in Table~\ref{tab:example}, $v_0$ as well as $C_v$ is also depicted. 

\begin{table}[!th]
\begin{center}
\caption{Examples of ordered filters.}\label{tab:example}
\renewcommand{\arraystretch}{1.7}
\begin{tabular}{ccccccc}
\toprule
Method            & $q_{\alpha}(\lambda)$  &  $C_q'$ & $C_q''$ & $v_0$ & $C_v$            & SVD required \\
\midrule
Spectral cut-off             & $\frac{1}{\lambda} \ind_{[\alpha, \infty)}(\lambda)$   & $1$  & $1$ & $\infty$ & $1$ & Yes \\
Tikhonov                       & $\frac{1}{\lambda+\alpha}$                                        & $1$       & $1$ & $1$ & $v^v(1-v)^{1-v}$ & No \\
$m$-iterated Tikhonov      & $\frac{(\lambda + \alpha)^m - \alpha^m}{\lambda(\lambda + \alpha)^m}$ & $m$ & $1$ &$m$ & $(v/m)^v(1-v/m)^{m - v}$ & No \\
Landweber ($\|T\| \le 1$)  & $\sum_{j = 0}^{\lfloor 1/\alpha \rfloor - 1} (1 - \lambda)^j$ & $1$ & $1$ & $\infty$ & $(v/e)^v$ & No \\
Showalter      			     & $\frac{1-\exp\left(-\frac{\lambda}{\alpha}\right)}{\lambda}$ & $1$  & $1$ & $\infty$ & $(v/e)^v$ & No \\
\bottomrule
\end{tabular}
\end{center}
\end{table}

For further reference, we collect the assumed properties of $f$ and $q_\alpha$ as follows:
\begin{assumption}\label{sc_conv_ass}
\begin{enumerate}
\item
The true solution $f$ satisfies $f \in \mathcal{W}_{\phi}(\rho)$ as in \eqref{eq:sc_general}.
\item 
The function $\phi$ is a qualification of the filter $q_\alpha$ as in~\eqref{eq:qc_general}.
\item\label{sc:c:3}
The function $\psi(x)\coloneqq x\phi^{-1}(\sqrt{x})$, with $\phi^{-1}$ being the inverse function of $\phi$, is convex.
\end{enumerate}
\end{assumption}
\begin{remark}
We stress that, being a standard assumption for convergence analysis, Assumption~\ref{sc_conv_ass}~(iii) actually imposes no restriction, since one can always work on a slightly larger source set $\mathcal{W}_{\widetilde\phi}$ with another index function $\widetilde{\phi}$ for which Assumption~\ref{sc_conv_ass}~(iii) is satisfied. {Note that function $\psi$ is strictly increasing, and its range is $\Rset_+$.}
\end{remark}

\subsection{Empirical prediction risk minimization}

As discussed in the Introduction, the optimal regularization parameter $\alpha$ will in general depend on $Y$, $\sigma$ and $f$, but the latter is unknown and hence this $\alpha$ cannot be realized in practice. Recall that we always assume $\sigma>0$ to be known. By means of the prediction risk $r\left(\alpha, f\right)$, the optimal $\alpha \in \mathcal A$ is given by
\begin{equation}\label{eq:alpha_o}
\alpha_o = \argmin_{\alpha \in \mathcal{A}} r(\alpha, f),
\end{equation}
{which is well-defined by similar arguments as in Appendix~\ref{app:alpha}.}
As a common remedy, we will try to estimate $\alpha_o$ from the observations $Y$ in~\eqref{eq:model} by minimizing an unbiased estimator of $r\left(\alpha,f\right)$, which can be derived as follows. 
Let us introduce the shorthand notation
\[
s_{\alpha}(\lambda) \coloneqq \lambda q_{\alpha}(\lambda),
\] 
Then we have
\begin{align}
r(\alpha,f) 
& = 
\E{\normY{T\left(q_{\alpha}(T^*T)T^*(T f + \sigma \xi) - f\right)}^2} \nonumber \\
& = \normY{(I - T q_{\alpha}(T^* T)T^* )T f}^2 + \sigma^2 \E{\normY{T q_{\alpha}(T^*T)T^*\xi}^2} \nonumber \\
&= \sum_{k=1}^{\infty} \lambda_k (1 - s_{\alpha}(\lambda_k))^2 f_k^2 + \sigma^2 \sum_{k=1}^{\infty} s_{\alpha}(\lambda_k)^2.\label{eq:pred_risk}
\end{align}
and furthermore
\begin{align*}
& \E{\snormY{T \hat{f}_{\alpha}}^2 - 2\inner{Y}{T\hat{f}_{\alpha}}_{\Y^*\times\Y}} \\
= &\, \E{\normY{T q_{\alpha}(T^*T) T^* (T f +\sigma \xi)}^2 - 2\inner{Tf+\sigma \xi}{T q_{\alpha}(T^*T) T^* (T f +\sigma \xi)}_{\Y^*\times\Y}} \\
 = &\, \sum_{k=1}^{\infty} \lambda_k (1-s_{\alpha}(\lambda_k))^2 f_k^2 - \sum_{k=1}^{\infty}\lambda_k f_k^2 + \sigma^2 \sum_{k=1}^{\infty} s_{\alpha}(\lambda_k)^2 - 2\sigma^2\sum_{k=1}^{\infty} s_{\alpha}(\lambda_k).
\end{align*}
Consequently for $\hat r\left(\alpha, Y\right)$ as in \eqref{eq:risk_estimator} we have 
\[
\E{\hat r\left(\alpha, Y\right)} = r(\alpha,f) - \sum_{k=1}^{\infty}\lambda_k f_k^2,
\]
i.e. up to a constant independent of $\alpha$, $\hat r\left(\alpha, Y\right)$ is an unbiased estimator of $r \left(\alpha, f\right)$. Hence we define
\begin{equation}\label{eq:alpha_pred} 
\alpha_{\mathrm{pred}} = \argmin_{\alpha\in\mathcal{A}} \hat{r}(\alpha, Y).
\end{equation}
{Note that $\alpha_{\mathrm{pred}}$ is measurable and almost surely well-defined, see Appendix~\ref{app:alpha} for details.}
\begin{remark}
For the clarity of our notation, we stress that $\alpha \mapsto r(\alpha, f)$ is a deterministic function, whereas $\alpha \mapsto \hat r \left(\alpha, Y\right)$ is a random function and consequently, ${\alpha_{\mathrm{pred}}\equiv\alpha_{\mathrm{pred}}(Y)}$ as in {\eqref{eq:alpha_pred}} is a random variable. {By $\E{\cdot}$ we always denote the expectation with respect to the data $Y$ or equivalently the noise $\xi$.} To obtain bounds for the estimator $\hat f_{\alpha_{\mathrm{pred}}}$ we will also need bounds for $r(\alpha_{\mathrm{pred}}, f)$, i.e. a deterministic function evaluated at a random variable. In particular, we stress that $\E{r(\alpha_{\mathrm{pred}}, f)} \neq \sE{\snormY{T \hat f_{\alpha_{\mathrm{pred}}} - T f}^2}.$ 
\end{remark}

Note that $\alpha_{\mathrm{pred}}$ is computable in practice as it only relies on the data, the forward operator {and the noise level (which is assumed to be known, see Section~\ref{s:con_out} for estimated noise levels)}. As discussed in the Introduction, there are many results available for the performance of $\hat{f}_{\alpha_{\mathrm{pred}}}$ measured by the prediction risk, among which the most precise one is due to \citet{k94} given below.  

\begin{theorem}[Deviation bound of prediction risk~\citep{k94}]\label{expBnd}
Assume the model~\eqref{eq:model}. Let $\hat{f}_{\alpha} = q_{\alpha}(T^*T)T^*Y$ with an ordered filter $q_{\alpha}$, $\alpha_o$ as in \eqref{eq:alpha_o}, and $\alpha_{\mathrm{pred}}$ as in \eqref{eq:alpha_pred}. Then there exist {universal} positive constants $C_\xi'$, $C_\xi''$ such that for all $x \ge 0$ and for all $f \in \X$,
\[
\Prob{{\frac{1}{\sigma^2}}\normY{T \hat{f}_{\alpha_{\mathrm{pred}}} - Tf}^2 -{\frac{1}{\sigma^2}}\normY{T \hat{f}_{\alpha_o} - Tf}^2 \ge x} \le C_\xi' \exp\left(-{C_\xi''}\min\biggl\{\sqrt{x},\, \frac{x}{\sqrt{{{r(\alpha_o, f)}/{\sigma^2}}}}\biggr\}\right), 
\]
which remains true when replacing ${\frac{1}{\sigma^2}}\|{T \hat{f}_{\alpha_{\mathrm{pred}}} - Tf}\|_\Y^2 -{\frac{1}{\sigma^2}}\|{T \hat{f}_{\alpha_o} - Tf}\|_\Y^2$ by ${\frac{1}{\sigma^2}}r(\alpha_{\mathrm{pred}}, f) - {\frac{1}{\sigma^2}}r(\alpha_o, f)$.
\end{theorem}

\begin{proof}
Recall that $\dim \X = \infty$, which, together with the injectivity of $T$,  implies $\dim \Y = \infty$. In fact, the assertion for finite dimensional $\Y$ follows directly from {Proposition 1 (i) and Theorem 1 in \cite{k94}} by chasing the dependency of the constants on the noise level $\sigma$. Concerning the key technical tools in Kneip's proof, we note that the Lemma~2 there actually holds for infinite sequences $a \in \ell^2$ and bounded linear trace operators $A: \ell^2 \to \ell^2$, and that the Lemma~3 there can be extended to infinite dimensional ordered linear smoothers, as long as they are Hilbert-Schmidt. Thus, the proof by~\citet{k94} carries over to the case that $\dim \Y = \infty$.
\end{proof}

The above theorem, in particular, implies \citep[cf.][Theorem 1]{cg14}
\[
\E{\normY{T \hat{f}_{\alpha_{\mathrm{pred}}} - Tf}^2} \le r(\alpha_o, f) + C \sigma \sqrt{r(\alpha_o, f)}\qquad \text{ for every }f\in \X, 
\]
which guarantees the order optimality of $\hat{f}_{\alpha_{\mathrm{pred}}}$ in terms of the prediction risk. 

\section{MISE estimates}\label{s:theory}

The section is devoted to the convergence analysis of $\hat{f}_{\alpha} = q_{\alpha}(T^*T)T^*Y$ with $\alpha = \alpha_{\mathrm{pred}}$ as in \eqref{eq:alpha_pred}. In what follows, we will prove that $\hat{f}_{\alpha_{\mathrm{pred}}}$ also possesses an optimality property in terms of the direct risk, i.e.~the mean integrated square error (MISE).

\subsection{Assumptions and merit discussions}

We start with some technical assumptions. As we have already seen in the previous section, many calculations involve summations over $\{\lambda_k\}_{k = 1}^{\infty}$, which can be formulated as Lebesgue-Stieltjes integrals with respect to 
\begin{equation*}
\Sigma(x)\coloneqq\#\{k:\lambda_k \ge x\}. 
\end{equation*}
Following~\citet{bhmr07}, we assume that $\Sigma$ can be approximated by a smooth function $S$ to avoid the difficulty caused by the non-smoothness of $\Sigma$:
\begin{assumption}\label{surrogate_ass}
\begin{enumerate}
\item
There exists a surrogate function $S\in\mathcal{C}^2\left((0,\infty)\right)$ of $\Sigma$ satisfying   
\begin{align*}
&\lim_{\alpha \searrow 0}S(\alpha) / \Sigma(\alpha) = 1,\\
& S'(\alpha) <0 \qquad \text{for all } \alpha > 0, \\
& \lim_{\alpha \nearrow \infty} S(\alpha) = \lim_{\alpha \nearrow \infty} S'(\alpha) = 0, \\
&  \lim_{\alpha \searrow 0} \alpha S(\alpha) = 0.
\end{align*}
\item
There exist constants $\alpha_1 \in (0, \lambda_1]$ and {$C_{S} > 0$ such that
\begin{align*}
\frac{1}{\alpha}\int_0^\alpha S(t) dt \le C_{S}S(\alpha) \qquad \text{ for all } \alpha \in (0, \alpha_1]. 
\end{align*}}
\item
There exists a constant $C_q>0$ such that 
\[
\int_{1}^{\infty} \Psi'(C_q x) \exp\left(-C_\xi''\sqrt{\frac{x}{2}}\right)\,\mathrm dx < \infty\qquad \text{ with } \Psi(x) \coloneqq \frac{x}{\left(S^{-1}\left(x\right)\right)^2}.
\]
Here $C_\xi''$ is as in Theorem~\ref{expBnd}.
\end{enumerate}
\end{assumption}
\begin{remark}{
Note that Assumption~\ref{surrogate_ass} (ii) holds true if there exists some $\tilde C_S \in (0,2)$ such that
\begin{align*}
-\alpha S'(\alpha) \text{ is integrable on } (0, \alpha_1], \qquad \text{ and } \qquad
\frac{S''(\alpha)}{-S'(\alpha)} \le \frac{\tilde C_{S}}{\alpha} \qquad \text{ for all } \alpha \in (0, \alpha_1],
\end{align*}
see Lemma 12 in \cite{bhmr07} for a proof. Thus, Assumption~\ref{surrogate_ass} (i) and (ii) are slightly weaker than Assumption 2} in~\cite{bhmr07} with a proper extension of $S$ to a larger domain $(0, \infty)$. We stress that these are rather weak requirements, and cover a wide range of situations~\citep[see e.g.][Section~5]{bhmr07}. The additional Assumption~\ref{surrogate_ass} (iii) is needed to control certain general moments of $r(\alpha_{\mathrm{pred}}, f)$ with the help of Theorem~\ref{expBnd}.
\end{remark}

The {next} assumption concerns the choice of regularization parameter $\alpha$ for (ordered) filters $q_{\alpha}$.
\begin{assumption}\label{para_filter_ass}
\begin{enumerate}
\item
For $\alpha \in \mathcal{A}$, the function $\lambda \mapsto s_{\alpha}(\lambda)$ is non-decreasing. 
\item
There exists $c_q > C_q^{-1/2}$ with $C_q$ as in Assumption \ref{surrogate_ass} (iii) such that 
\begin{equation}\label{eq:parametrization}
s_{\alpha}(\alpha) = \alpha q_{\alpha} (\alpha) \ge c_q \qquad \text{ as } \alpha \searrow 0.
\end{equation}
\end{enumerate}
\end{assumption}
\begin{remark}
It is easy to see that the above condition with some $c_q>0$ is satisfied by all the regularization methods in Table~\ref{tab:example}, which indicates that Assumption~\ref{para_filter_ass} is fairly general.  In particular, Assumption~\ref{para_filter_ass} (ii) can be understood as a parametrization condition. For instance, the Tikhonov method with re-parametrization $\alpha \mapsto \sqrt{\alpha}$, i.e. $q_{\alpha}(\lambda) = 1/(\sqrt{\alpha} + \lambda)$, still defines an ordered filter, but does not satisfy \eqref{eq:parametrization} anymore. The condition $c_q > C_q^{-1/2}$ requires some compatibility of the ill-posedness with the parametrization of the filter. {Note that the qualification condition~\eqref{eq:qc_general} implies that $s_{\alpha}(\alpha) \le 1$ for $\alpha > 0$, so if $C_q > 1$ then the relation $c_q > C_q^{-1/2}$ is automatically satisfied.}
\end{remark}

Under these assumptions, upper bounds for $\hat f_\alpha$ with an a-priori choice of $\alpha$ have been proven by \citet{bhmr07}:
\begin{theorem}[A-priori parameter choice~\citep{bhmr07}]\label{apriori_choice}
Consider the model~\eqref{eq:model}, and let $\hat{f}_{\alpha} \coloneqq q_{\alpha}(T^*T)T^*Y$ with a filter $q_{\alpha}$, and {suppose that Assumption~\ref{surrogate_ass} (i)-(ii) holds true}. Let also $\alpha_*$ satisfy 
\begin{equation}\label{eq:apriori_choice}
\alpha_* \phi(\alpha_*)^2 = \sigma^2 S(\alpha_*). 
\end{equation}
\begin{enumerate}
\item
If Assumption \ref{sc_conv_ass} (ii) holds, {there is a constant $C_1$ depending only on $\rho$, $C_{\phi}$, $C_q'$, $C_q''$ and $C_S$ such that} 
\[
\sup_{f \in \mathcal{W}_{\phi}(\rho)}\E{\normX{\hat f_{\alpha_*} - f}^2} \le C_1 \phi(\alpha_*)^2 = C_1\sigma^2 \frac{S(\alpha_*)}{\alpha_*} \qquad \text{ as }\sigma \searrow 0.  
\]
\item
If $\tilde{\phi}(t) \coloneqq \sqrt{t}\phi(t)$ is a qualification of the filter $q_\alpha$, {namely,
\begin{equation}\label{eq:qctilde}
\sup_{\lambda \in [0, \lambda_1]} \tilde\phi(\lambda) \abs{1 - \lambda q_{\alpha}(\lambda)} \le C_{\tilde\phi} \tilde\phi(\alpha)  \qquad \text{ for all } \alpha \in \mathcal{A},
\end{equation} 
then there is a constant $C_2$ depending only on $\rho$, ${C}_{\tilde\phi}$, $C_q'$, $C_q''$ and $C_S$ such that}
\[
\sup_{f \in \mathcal{W}_{\phi}(\rho)}\E{\normY{T \hat f_{\alpha_*} - T f}^2} \le C_2\alpha_* \phi(\alpha_*)^2 = C_2 \sigma^2 S(\alpha_*) \qquad \text{ as }\sigma \searrow 0.  
\]
\end{enumerate}
\end{theorem}
{
\begin{remark}
Actually, the upper bounds on the risk in Theorem~\ref{apriori_choice} consist of a bias part and a variance part. The front constant in the bias part depends only on $\rho$ and $C_{\phi}$ (or $C_{\tilde\phi}$), while that in the variance part depends only on $C_q'$, $C_q''$ and $C_S$. Moreover, it is worth noting that \eqref{eq:qctilde} in particular implies \eqref{eq:qc_general}, i.e., whenever $\tilde{\phi}(t) \coloneqq \sqrt{t}\phi(t)$ is a qualification of the filter $q_\alpha$, then also $\phi$ is a qualification of the filter $q_\alpha$ and Assumption \ref{sc_conv_ass} (ii) is satisfied. 
\end{remark}
} 
\subsection{Oracle inequality}\label{ss:oi}
We are now in position to derive an oracle inequality in the general form of \eqref{eq:goie} for the empirical parameter choice $\alpha = \alpha_{\mathrm{pred}}$.
\begin{theorem}[Oracle inequality]\label{oracle_ineq}
Assume the model~\eqref{eq:model}. Let $\hat{f}_{\alpha} \coloneqq q_{\alpha}(T^*T)T^*Y$ with an ordered filter $q_{\alpha}$, and Assumptions~\ref{sc_conv_ass}, \ref{surrogate_ass} and \ref{para_filter_ass} hold. Let also $\alpha_o$ be given by~\eqref{eq:alpha_o}, and $\alpha_{\mathrm{pred}}$ by~\eqref{eq:alpha_pred}. {Then there are positive constants $C_1, C_2$ and $C_3$, independent of $f$, $\alpha$ and $\sigma$, such that 
\begin{equation}\label{eq:oracle_ineq}
\E{\normX{\hat{f}_{\alpha_{\mathrm{pred}}} - f}^2} \le  \rho^2\psi^{-1}\biggl(\sigma^2\Bigl(\frac{2}{\rho^2}\gamma_\sigma + C_1\Bigr)\biggr) + \sigma^2 C_3 \biggl(\frac{\gamma_\sigma + \sqrt{\gamma_\sigma}}{S^{-1} \left(2 C_q \gamma_\sigma \right)} + C_2\biggr)\qquad \text{with }\gamma_\sigma\coloneqq \frac{r(\alpha_o, f)}{\sigma^2}. 
\end{equation}}
\end{theorem}

\begin{remark}
{Despite the fact that \eqref{eq:oracle_ineq} is not an oracle inequality in the strict sense as discussed e.g. by \cite{cgpt02}, we still call \eqref{eq:oracle_ineq} an oracle inequality as it relates the direct risk under $\alpha_{\mathrm{pred}}$ with the weak \emph{oracle} risk $r(\alpha_o, f)$. We emphasize that this is in line with \cite{cg14} and refer to \cite{w17} for further discussion. 
Moreover, we point out that constants $C_1, C_2$ in Theorem~\ref{oracle_ineq} are in fact universal, while $C_3$ depends only on $C_q', C_q'', C_q, C_S$, $c_q$ and the operator $T$. 
}
\end{remark}

The proof of Theorem~\ref{oracle_ineq} is based on a general strategy together with technical lemmata, and is postponed to Section \ref{sec:proofs}. To ease the understanding of this paper, we will now start with conclusions from Theorem \ref{oracle_ineq}. 

\subsection{Convergence rates and examples}
The derived oracle inequality in Theorem~\ref{oracle_ineq} readily provides error estimates for the estimator $\hat f_{\alpha_{\mathrm{pred}}}$ given proper upper bounds of the oracle prediction risk $r(\alpha_o, f)$ as in Theorem \ref{apriori_choice}. 

\begin{theorem}[Convergence rates]\label{conv_rate}
Assume the same setting as in Theorem~\ref{oracle_ineq}, and additionally that $\tilde{\phi}(t) = \sqrt{t}\phi(t)$ is a qualification of the filter $q_\alpha$. Let $\alpha_*$ be given by~\eqref{eq:apriori_choice}. Then as $\sigma \searrow 0$ we have
\begin{equation*}
\sup_{f \in \mathcal{W}_{\phi}(\rho)}\E{\normX{\hat f_{\alpha_{\mathrm{pred}}} - f}^2} \le C_1\phi(\alpha_*)^2 + C_3\frac{\alpha_* \phi(\alpha_*)^2}{S^{-1}\left(C_2C_q \frac{\alpha_*\phi(\alpha_*)^2}{\sigma^2}\right)} = C_1\sigma^2\frac{S(\alpha_*)}{\alpha_*} + C_3\frac{\sigma^2 S(\alpha_*)}{S^{-1}\left(C_2C_q S(\alpha_*)\right)}
\end{equation*}{
for some constants $C_1, C_2, C_3 > 0$ independent of $\sigma$.}
\end{theorem}
\begin{proof}{
From Theorem~\ref{oracle_ineq}, it follows that 
$$
\sup_{f \in \mathcal{W}_{\phi}(\rho)}\E{\normX{\hat f_{\alpha_{\mathrm{pred}}} - f}^2}\le \sup_{f \in \mathcal{W}_{\phi}(\rho)}\left(\rho^2\psi^{-1}\Bigl(\frac{2}{\rho^2} r(\alpha_o, f)+ C_1\sigma^2\Bigr) + C_3 \biggl(\frac{r(\alpha_o, f) + \sigma\sqrt{r(\alpha_o, f)}}{S^{-1} \left(2 C_q \frac{r(\alpha_o, f)}{\sigma^2} \right)} + C_2 \sigma^2\biggr)\right),
$$
where $C_1$, $C_2$ are universal, and $C_3$ depends only on  $C_q', C_q'', C_q, C_S$, $c_q$ and the operator $T$. By Theorem~\ref{apriori_choice}~(ii), there is a constant $C_4 \ge \rho^2$ depending only on $\rho$, ${C}_{\tilde\phi}$, $C_q'$, $C_q''$ and $C_S$ such that
$$
\sup_{f \in \mathcal{W}_{\phi}(\rho)}r(\alpha_o, f) \le \sup_{f \in \mathcal{W}_{\phi}(\rho)} r(\alpha_*, f) \le C_4 \alpha_*\phi(\alpha_*)^2 = C_4\sigma^2S(\alpha_*)\qquad\text{for sufficiently small }\sigma. 
$$
Note that by definition $\alpha_* \searrow 0$ as $\sigma \searrow 0$, so 
$$
\frac{\alpha_*\phi(\alpha_*)^2}{\sigma^2} = S(\alpha_*) \nearrow \infty\qquad \text{ as } \sigma \searrow 0.  
$$
Recall that $\psi(x) = x\phi^{-1}(\sqrt{x})$ is strictly increasing and convex, and $\psi(0) = 0$. Thus, $\psi^{-1}$ is strictly increasing, and $\psi^{-1}(cx) \le c\psi^{-1}(x)$ for any $c\ge 1$. Note also that $S^{-1}$ is strictly decreasing. Thus,  for small enough~$\sigma$
\begin{align*}
&\rho^2\psi^{-1}\Bigl(\frac{2}{\rho^2} r(\alpha_o, f)+ C_1\sigma^2\Bigr) + C_3 \biggl(\frac{r(\alpha_o, f) + \sigma\sqrt{r(\alpha_o, f)}}{S^{-1} \left(2 C_q \frac{r(\alpha_o, f)}{\sigma^2} \right)} + C_2 \sigma^2\biggr) \\
\le \,& \rho^2\psi^{-1}\Bigl(\frac{3}{\rho^2}C_4 \alpha_*\phi(\alpha_*)^2 \Bigr) + \frac{2C_3C_4\alpha_*\phi(\alpha_*)^2}{S^{-1} \left(2 C_q C_4\frac{\alpha_*\phi(\alpha_*)^2}{\sigma^2} \right)} \\
\le \, & 3C_4 \psi^{-1}\bigl(\alpha_*\phi(\alpha_*)^2\bigr) +  \frac{2C_3C_4\alpha_*\phi(\alpha_*)^2}{S^{-1} \left(2 C_q C_4\frac{\alpha_*\phi(\alpha_*)^2}{\sigma^2} \right)} = 3C_4\phi(\alpha_*)^2 +  \frac{2C_3C_4\alpha_*\phi(\alpha_*)^2}{S^{-1} \left(2 C_q C_4\frac{\alpha_*\phi(\alpha_*)^2}{\sigma^2} \right)},
\end{align*}
which holds uniformly over $f \in \mathcal{W}_{\phi}(\rho)$, and thus concludes the proof.}
\end{proof}

\begin{remark}\label{rmk:conv_rate}
If, in addition, there is a constant {$\tilde{C}>0$} such that
\[
S({\tilde{C}} x) \ge C_2 C_q S(x) \qquad \text{ for all } x > 0,
\]
then it follows from Theorem~\ref{conv_rate} that
\[
\sup_{f \in \mathcal{W}_{\phi}(\rho)}\E{\normX{\hat f_{\alpha_{\mathrm{pred}}} - f}^2} \le C \phi(\alpha_*)^2 = C\sigma^2\frac{S(\alpha_*)}{\alpha_*}
\]
with some $C>0$ independent of $\alpha_*$ and $\sigma$. This coincides with the convergence rate we obtain under an a-priori parameter choice in Theorem~\ref{apriori_choice}, which turns out to be order optimal in most cases, see e.g.~Section~\ref{ss:example_mild} and~\citet{bhmr07}. 

We note that there are two additional important assumptions for a-posteriori parameter choice $\alpha_{\mathrm{pred}}$ compared to the a-priori choice (cf.~Theorem~\ref{apriori_choice} (i) and Theorem~\ref{conv_rate}). The one is Assumption~\ref{surrogate_ass} (iii), which concerns the control of general moments of $r(\alpha_{\mathrm{pred}}, f)$, due to the randomness of $\alpha_{\mathrm{pred}}$. The other is that not only $\phi$ but also $\tilde{\phi}(t) = \sqrt{t}\phi(t)$ is a qualification of the filter $q_\alpha$. The latter seems to be typical for parameter choice rule relying on residuals in image space, as e.g.~for the discrepancy principle in case of deterministic inverse problems, see~\citet[][Section~4.3]{ehn96} and~\citet{mp06}, or for generalized cross-validation (GCV), see \citet{l93} and \citet{v02}. Still we stress that the qualification assumption plays no role in the proof of the oracle inequality (Theorem~\ref{oracle_ineq}), and it only kicks in for the derivation of convergence rates for $\hat f_{\alpha_{\mathrm{pred}}}$ through $r(\alpha_o, f)$ in the convergence analysis under the a-priori parameter choice. 
\end{remark}

\subsubsection{Mildly ill-posed problems}\label{ss:example_mild}

We now consider a particular mildly ill-posed problem. More precisely, we assume
\begin{assumption}\label{mildIP_ass}
\begin{enumerate}
\item
Polynomial decay of eigenvalues of $T^*T$
\[
\lambda_k =  C_a k^{-a}\qquad k = 1,2, \ldots\qquad \text{ with some } a > 1 \text{ and } C_a > 0;
\]
\item
Smoothness of the truth
\[
f \in {\mathcal{S}_{b}} \coloneqq \left\{f \in \X:\sum_{k=1}^\infty w_k f_k^2  \le 1\right\} \qquad \text{ with }w_k = C_b k^ b\quad k= 1,2,\ldots,
\] 
for some positive constants $b$ and $C_b$. 
\end{enumerate}
\end{assumption}
\noindent In the above assumption, the requirement of $a > 1$ is to ensure that the forward operator $T$ is Hilbert-Schmidt, and {the smoothness class $\mathcal{S}_{b}$ is equivalent to the H{\"o}lder type source condition of order $b/(2a)$, more precisely, $\mathcal{S}_b \equiv \mathcal{W}_{b/(2a)}\bigl(C_a^{-b/a}C_b^{-1}\bigr)$ in~\eqref{eq:sc_hoelder}.} 

In this simple setting the convergence rates of ordered filter based methods with empirical parameter choice $\alpha_{\mathrm{pred}}$ can be explicitly computed. 
\begin{corollary}[Order optimality]\label{order_optimal}
Assume the model~\eqref{eq:model}. Let $\hat{f}_{\alpha} \coloneqq q_{\alpha}(T^*T)T^*Y$ with an ordered filter $q_{\alpha}$, and Assumptions~\ref{para_filter_ass} and~\ref{mildIP_ass} with any $c_q>0$ hold. If the qualification index $v_0$ in~\eqref{eq:qc_hoelder} of the ordered filter satisfies $v_0 \ge b/(2a) + 1/2$, then there are positive constants $C_1$ and $C_2$ {independent of $\sigma$} such that 
\[
\sup_{f \in \mathcal{S}_b}\E{\normX{\hat{f}_{\alpha_{\mathrm{pred}}} -f}^2} \le C_1 \inf_{\hat{f}} \sup_{f \in \mathcal{S}_b} \E{\normX{ \hat{f} - f}^2} \le C_2 \sigma^{\frac{2b}{a+b+1}}\qquad \text{ as } \sigma \searrow 0.
\] 
The infimum above is taken over all possible estimators $\hat f$ (including both linear and nonlinear ones).
\end{corollary}
\begin{proof}
{
Note that $\phi(x) = x^{b/(2a)}$ and then $\psi(x) = x\phi^{-1}(\sqrt{x}) = x^{(a+b)/b}$, so Assumption~\ref{sc_conv_ass} is satisfied. 
Define $S(\alpha) \coloneqq \bigl(\alpha/C_a \bigr)^{-1/a}$. Elementary calculation shows that $S(\alpha)$ satisfies Assumption~\ref{surrogate_ass} (i) and (ii).}  Since $\Psi(x) = x/ \bigl(S^{-1}(x)\bigr)^2 =  C_a^{-2} x^{1+2a}$, Assumption~\ref{surrogate_ass} (iii) clearly holds. By definition~\eqref{eq:apriori_choice}, it follows that $\alpha_* = C_a^{1/(1+a+b)}\sigma^{2a/(1+a+b)}$. Thus, by Theorem~\ref{conv_rate} and Remark~\ref{rmk:conv_rate}, we have
\[
\sup_{f \in \mathcal{S}_b}\E{\normX{\hat{f}_{\alpha_{\mathrm{pred}}} -f}^2} \le C_1\phi(\alpha_*)^2 = C_1 C_a^{\frac{b}{a(1+a+b)}}(\sigma^2)^{\frac{b}{1+a+b}}\qquad \text{ as }\sigma \searrow 0,
\]
for some positive constant $C_1$. 

Further, it is well-known~\citep[see~e.g.][]{p80,gk99,DiMa17}
\[
 \inf_{\hat{f}} \sup_{f \in \mathcal{S}_b} \E{\normX{ \hat{f} - f}^2} \ge C_2 (\sigma^2)^{\frac{b}{a+b+1}}\qquad \text{ for some } C_2 > 0.
\]
This concludes the proof. 
\end{proof}

\begin{remark}
The above proposition in particular implies that all the methods in Table~\ref{tab:example} are order optimal with $\alpha = \alpha_{\mathrm{pred}}$ when $b/(2a)+1/2\le v_0$ in the minimax sense. This reproduces the result in~\citet{cg14} for the spectral cut-off method as a special case. 

Note that the parameter choice $\alpha_{\mathrm{pred}}$ depends only on the data, and is completely independent of the unknown truth. Thus, the ordered filter based regularization methods with $\alpha_{\mathrm{pred}}$ automatically adapt to the unknown smoothness of the truth, and achieve the best possible rates up to a constant. In other words, the ordered filter based regularization methods with parameter choice  $\alpha_{\mathrm{pred}}$  are \emph{adaptively minimax optimal}~\citep[cf.][]{p80} over a range of smoothness classes, and the adaptation range is determined by the qualification index of the filter and the smoothing property of the forward operator. Importantly, we point out that the order optimality here is in sharp contrast to the Lepski\u{\i}-type balancing principle~\citep{l90}, where one typically loses a log-factor in the asymptotic convergence rates. 

In addition, we stress again the price we pay for a-posteriori parameter choice $\alpha_{\mathrm{pred}}$ is a stronger qualification assumption $b/(2a)+1/2\le v_0$, as the convergence rates of direct risk for a-priori parameter choice only asks for $b/(2a) \le v_0$ by Theorem~\ref{apriori_choice} (i), see also Remark~\ref{rmk:conv_rate}.
\end{remark}

\subsubsection{Exponentially ill-posed problems}\label{ss:example_serious}
Next we consider an exponentially ill-posed setting: The eigenvalues of $T^*T$ satisfy
\[
\lambda_k \asymp \exp\left(- \mu k^{a}\right)\qquad k = 1,2,\ldots\qquad\text{ for some }  a, \mu > 0,
\]
and the smoothness of the truth is characterized by~\eqref{eq:sc_general} with
\[
\phi(x)\coloneqq x^{b} (-\log x)^c\qquad \text{for some } b > 0, c \in \Rset, \text{ or } b = 0, c > 0.
\]
That is, we assume that the truth lies in $\left\{f \in \X:\sum_{k=1}^\infty w_k f_k^2  \le 1\right\}$ with $w_k \asymp k^{-2ac}\exp(2\mu b k^a)$.

It is easy to see that the assumptions of Theorem~\ref{oracle_ineq} are satisfied for all the regularization methods listed in Table~\ref{tab:example} provided that $0< a <1/2$, or $a = 1/2$ and $\mu$ is sufficiently small. Hence, if in addition $b +1/2 < v_0$, or $b = v_0, c \le 0$ (recall that $v_0$ is the qualification index), we can obtain certain error bounds by means of Theorem~\ref{conv_rate}. However, it turns out that our bounds are too rough to guarantee order optimality. In fact, in case of $b=0$, the error bound on the right-hand side of \eqref{eq:oracle_ineq} even diverges as the noise level $\sigma$ tends to $0$.
In summary, our oracle inequality is applicable for exponentially ill-posed problems as discussed here, but is not strong enough to derive rates of convergence or even show optimality of the investigated parameter choice $\alpha_{\mathrm{pred}}$. {We refer to \cite{w17} for numerical simulations in exponentially ill-posed examples, in the view of which it seems questionable if the parameter choice rule under investigation still yields optimal results, being in line with the findings by \citet{lpbbbdw17}.}

\section{A general methodology for proving the oracle inequality}\label{sec:proofs} 

To prove our oracle inequality \eqref{eq:oracle_ineq}, we will proceed as follows. First, we obtain bounds for general moments of the prediction risk (cf. Corollary \ref{ctrMoment}). Secondly, we proceed with a standard estimate for the (deterministic) bias of the estimator $\hat f_{\alpha}$ under the smoothness Assumption \ref{sc_conv_ass} (cf. Lemma \ref{source_cond}). Finally, we prove a comparison lemma on the variance terms (cf. Lemma \ref{compare}). Putting these three ingredients together, the proof of Theorem~\ref{oracle_ineq} is then straightforward. More importantly, we stress that, if these three ingredients are given for some parameter choice rule $\bar \alpha$, an oracle inequality similar to \eqref{eq:oracle_ineq} can be derived.

\subsection{Main ingredients}

We start with bounds for general moments of the prediction risk.
\begin{corollary}[General moments of prediction risk]\label{ctrMoment}
Assume the same setting as Theorem~\ref{expBnd}, and let $\Psi : \Rset_+ \to \Rset_+$ be an increasing and continuously differentiable function with $\Psi(0) = 0$ such that
\begin{equation}\label{eq:tail_cond}{
C_{\psi}\coloneqq C_\xi'\int_{0}^{\infty} \Psi'(x) \exp\left(-C_\xi''\sqrt{\frac{x}{2}}\right)dx < \infty}
\end{equation}
with $C_\xi', C_\xi''$ being the same as in Theorem~\ref{expBnd}. Then 
\[
\E{\Psi\left(\frac{r(\alpha_{\mathrm{pred}},f)}{\sigma^2}\right)} \le  \Psi \left(\frac{2r(\alpha_o, f)}{\sigma^2}\right) + C_{\Psi} \qquad \text{ for all } f \in \X.
\]
\end{corollary}

\begin{proof} Let $\gamma_\sigma \coloneqq r(\alpha_o, f)/\sigma^2$ as in~\eqref{eq:oracle_ineq}. Then
\begin{align*}
\E{\Psi\left(\frac{r(\alpha_{\mathrm{pred}}, f)}{\sigma^2}\right)} = & \int_{0}^{\infty} \Psi'(x) \Prob{\frac{r(\alpha_{\mathrm{pred}}, f)}{\sigma^2}  \ge x}\,\mathrm dx &&\text{[by Fubini's theorem]}\\
\le & \int_{0}^{2\gamma_\sigma} \Psi'(x)\,\mathrm dx + C_\xi'\int_{2\gamma_\sigma}^{\infty}\Psi'(x)\exp\left(-C_\xi''\sqrt{x -\gamma_\sigma}\right)\,\mathrm dx &&\text{[by Theorem~\ref{expBnd}]}\\
\le & \Psi\left({2\gamma_\sigma}\right) + C_\xi'\int_{2\gamma_\sigma}^{\infty}\Psi'(x)\exp\left(-C_\xi''\sqrt{\frac{x}{2}}\right)\,\mathrm dx \\
\le & \Psi\left(\frac{2 r(\alpha_o, f)}{\sigma^2}\right) + C_{\Psi}. 
\end{align*}
This concludes the proof. 
\end{proof}
\begin{remark}\label{rmk:moments}
A simple example is $\Psi(x) \coloneqq x^{\theta}$ for some $\theta > 0$. By Corollary~\ref{ctrMoment}, it leads to 
\[
\E{r(\alpha_{\mathrm{pred}},f)^\theta} \le 2^{\theta}r(\alpha_o, f)^{\theta} + C_{\theta} \sigma^{2\theta},
\]
where the second term on the right hand side is typically negligible compared to the first as $\sigma$ tends to $0$. {Note that condition~\eqref{eq:tail_cond} only requires
$$
\int_{c}^{\infty} \Psi'(x) \exp\left(-C_\xi''\sqrt{\frac{x}{2}}\right)dx < \infty \qquad \text{ for some contant } c > 0. 
$$}
\end{remark}

Based on the smoothness Assumption \ref{sc_conv_ass}, we can prove the following estimate for the bias of $\hat f_{\alpha}$:
\begin{lemma}[Source condition]\label{source_cond}
Under Assumption~\ref{sc_conv_ass} (i) and (iii), it holds that
\[
\sum_{k = 1}^{\infty}  (1-s_{\alpha}(\lambda_k))^2f_k^2  \le \rho^2\psi^{-1} \left(\frac{1}{{\rho^2}}\sum_{k=1}^{\infty} \lambda_k  (1-s_{\alpha}(\lambda_k))^2f_k^2 \right),
\]
with $\rho$ {as in \eqref{eq:sc_general}}.
\end{lemma}
\begin{proof}
Based on spectral analysis, condition~\eqref{eq:sc_general} can be equivalently written as
\[
f_k = \phi(\lambda_k) w_k, \quad k = 1,2,\ldots, \qquad \text{ with } \sum_{k = 1}^{\infty} w_k^2 \le \rho^2.
\]
Thus, $f_k^2 \le \rho^2 \phi(\lambda_k)^2$ and then $\phi^{-1}(\sqrt{\rho^{-2}f_k^2}) \le \lambda_k$.  {Note that $\psi(0) = 0$ and $\psi(cx) \ge c\psi(x)$ for any $c \ge 1$. By the convexity of $\psi$ and Jensen's inequality, we have for any $n \in \mathbb N$
\begin{align*}
\psi\Bigl(\sum_{k = 1}^{n} \rho^{-2} f_k^2 \left(1-s_{\alpha}(\lambda_k)\right)^2\Bigr) & \le\, \psi\Bigl(\sum_{k = 1}^{n} \rho^{-2} f_k^2 \frac{\left(1-s_{\alpha}(\lambda_k)\right)^2}{1 + \sum_{i = 1}^n\left(1-s_{\alpha}(\lambda_i)\right)^2}\Bigr)\Bigl(1 + \sum_{i = 1}^n\left(1-s_{\alpha}(\lambda_i)\right)^2\Bigr) \\
& \le \,\sum_{k = 1}^{n} \psi\bigl( \rho^{-2} f_k^2\bigr) \frac{\left(1-s_{\alpha}(\lambda_k)\right)^2}{1+\sum_{i = 1}^n\left(1-s_{\alpha}(\lambda_i)\right)^2}\Bigl(1+\sum_{i = 1}^n\left(1-s_{\alpha}(\lambda_i)\right)^2\Bigr) \\
& = \,
 \sum_{k=1}^{n} \psi\bigl(\rho^{-2} f_k^2\bigr)\left(1 - s_{\alpha}(\lambda_k)\right)^2 
 \le \rho^{-2} \sum_{k = 1}^{n} \lambda_k f_k^2 \left(1 - s_{\alpha}(\lambda_k)\right)^2\\
& \le\, \rho^{-2} \sum_{k = 1}^{\infty} \lambda_k f_k^2 \left(1 - s_{\alpha}(\lambda_k)\right)^2.
\end{align*}
The assertion follows by applying $\psi^{-1}$ to the above inequality and letting $n\nearrow \infty$.}
\end{proof}

Finally, we need certain comparison relations, which are {used} for bounding the variance part of the risk: 
\begin{lemma}[Comparison]\label{compare}
Let $q_{\alpha}$ be a filter, and Assumptions~\ref{surrogate_ass} and~\ref{para_filter_ass} hold. Then there are positive constants $C_1$ and $C_2$, {depending only on $C_q', C_q'', C_q, C_S$, $c_q$ and the operator $T$,} such that for every $\alpha \in \mathcal{A}$
\[
\sum_{k=1}^{\infty} \lambda_k q_{\alpha}(\lambda_k)^2 \le C_1 \Psi_1 \left( C_q \sum_{k=1}^{\infty} s_{\alpha}(\lambda_k)^2\right)\quad \text{and}\quad \sum_{k=1}^{\infty} \lambda_k^2 q_{\alpha}(\lambda_k)^4 \le C_2 \Psi_2 \left(C_q\sum_{k=1}^{\infty} s_{\alpha}(\lambda_k)^2\right), 
\]
{with $C_q', C_q''$ in Definition~\ref{def:order_filter}, $C_q, C_S$ in Assumption~\ref{surrogate_ass}, $c_q$ in Assumption~\ref{para_filter_ass}, and functions $\Psi_1$ and $\Psi_2$ by}
\begin{equation}\label{eq:Psis}
\Psi_1\left(x\right) \coloneqq \frac{x}{S^{-1} \left(x\right)}, \qquad \Psi_2\left(x\right) \coloneqq \frac{x}{\left(S^{-1} \left(x\right)\right)^2}, \qquad \text{ for every } x>0.
\end{equation}
\end{lemma}

\begin{proof} {
By Assumptions~\ref{surrogate_ass} (i) and~\ref{para_filter_ass} (ii), and $c_q^2C_q > 1$, there are $\alpha_0$, $\alpha_0 \in (0,  \min\{1, \alpha_1\})$, with $\alpha_1$ in Assumption~\ref{surrogate_ass} (ii), and some constant $\delta\in(0,1)$ such that for every $\alpha \le \alpha_0$ it holds that 
\begin{equation}\label{eq:smAlp}
s_{\alpha}(\alpha)\ge \delta c_q,\qquad \text{ and }\qquad \frac{1}{\delta^2 c_q^2C_q} \le \frac{\Sigma(\alpha)}{S(\alpha)} \le 2.
\end{equation}
We consider two separate cases:}

\textbf{Case I:} $\alpha \le \alpha_0$. Then   
\begin{align}
\sum_{k = 1}^{\infty} s_{\alpha}(\lambda_k)^2 & = -\int_0^{\infty} t^2 q_{\alpha}(t)^2\,\mathrm d\Sigma(t) \ge -\int_{\alpha}^{\infty} t^2 q_{\alpha}(t)^2\,\mathrm d\Sigma(t) && \nonumber \\
& \ge -\int_{\alpha}^{\infty} \alpha^2 q_{\alpha}(\alpha)^2\,\mathrm d\Sigma(t) && \text{[by Assumption~\ref{para_filter_ass} (i)]} \nonumber \\
& \ge  -{\delta^2}c_q^2 \int_{\alpha}^{\infty}\,\mathrm d\Sigma(t) = {\delta^2}c_q^2 \Sigma(\alpha)   && \text{[by  \eqref{eq:smAlp}]} \nonumber \\
& {\ge\frac{1}{C_q} S(\alpha).} &&{\text{[by \eqref{eq:smAlp}]}}  \label{eq:loBnd_s}
\end{align}
{
Denote $\tilde C_q \coloneqq \max\{C_q', C_q''\}$. Then}
\begin{align*}
\sum_{k=1}^{\infty} \lambda_k q_{\alpha}(\lambda_k)^2 & = - \int_0^{\infty} t q_{\alpha}(t)^2\,\mathrm d\Sigma(t) && \\
& \le {\tilde C_q^2}\left(-\int_0^{\alpha}\frac{t}{\alpha^2}\,\mathrm d\Sigma(t) - \int_{\alpha}^{\infty} \frac{1}{t}\,\mathrm d\Sigma(t)\right) && \text{[by Definition~\ref{def:order_filter} (i)]} \\
& \le {\tilde C_q^2}\left(-\frac{t}{\alpha^2} \Sigma(t)\Big\vert_0^{\alpha} + \frac{1}{\alpha^2} \int_0^{\alpha} \Sigma(t)\,\mathrm dt + \frac{1}{\alpha} \Sigma(\alpha) \right) =  \frac{{\tilde C_q^2}}{\alpha^2} \int_0^{\alpha} \Sigma(t)\,\mathrm dt&& \\
& \le \frac{{2\tilde C_q^2}}{\alpha^2}\int_0^{\alpha} S(t)\,\mathrm dt && \text{[by  \eqref{eq:smAlp}]}\\ 
& \le {2\tilde C_q^2C_S}\frac{S(\alpha)}{\alpha}, && \text{[by~Assumption~\ref{surrogate_ass} (ii)]}
\end{align*}
and similarly
\begin{align*}
\sum_{k=1}^{\infty} \lambda_k^2 q_{\alpha}(\lambda_k)^4 & = - \int_0^{\infty} t^2 q_{\alpha}(t)^4\,\mathrm d\Sigma(t) && \\
& \le {\tilde C_q^4}\left(-\int_0^{\alpha}\frac{t^2}{\alpha^4}\,\mathrm d\Sigma(t) - \int_{\alpha}^{\infty} \frac{1}{t^2}\,\mathrm d\Sigma(t)\right) && \text{[by Definition~\ref{def:order_filter} (i)]} \\
& \le {\tilde C_q^4}\left(-\frac{t^2}{\alpha^4} \Sigma(t)\Big\vert_0^{\alpha} + \frac{1}{\alpha^4} \int_0^{\alpha} 2t \Sigma(t)\,\mathrm dt + \frac{1}{\alpha^2} \Sigma(\alpha) \right)  && \\
& \le  \frac{{2\tilde C_q^4}}{\alpha^3} \int_0^{\alpha} \Sigma(t)\,\mathrm dt\le \frac{{4 \tilde C_q^4}}{\alpha^3}\int_0^{\alpha} S(t)\,\mathrm dt  && \text{[by \eqref{eq:smAlp}]}\\
& \le {{4\tilde C_q^4}{C_S}}\frac{S(\alpha)}{\alpha^2}. && \text{[by Assumption~\ref{surrogate_ass} (ii)]}
\end{align*}
These together with~\eqref{eq:loBnd_s} prove the assertion {if $\alpha \le \alpha_0$.}

\textbf{Case II:} $\alpha > \alpha_0$.Then by Assumption~\ref{para_filter_ass} it holds that
\[
\sum_{k = 1}^{\infty} s_{\alpha}(\lambda_k)^2 {\ge \sum_{k = 1}^{\infty} s_{\alpha_0}(\lambda_k)^2,\qquad \text{a constant }\in\Rset_+.} 
\]
By Definition~\ref{def:order_filter} (i) we have{
\begin{align*}
&\sum_{k = 1}^{\infty} \lambda_k q_{\alpha}(\lambda_k)^2 \le\frac{(C'_q)^2}{\alpha^2} \sum_{k = 1}^{\infty} \lambda_k \le \frac{(C'_q)^2}{\alpha_0^2} \Tr(T^*T),\\ 
\text{ and }\quad &\sum_{k = 1}^{\infty} \lambda_k^2 q_{\alpha}(\lambda_k)^4 \le \frac{C''_q (C'_q)^3}{\alpha^3} \sum_{k = 1}^{\infty} \lambda_k \le \frac{C''_q (C'_q)^3}{\alpha_0^3}\Tr(T^*T), 
\end{align*}}
where both upper bounds are constants in $\Rset_+$. Thus, the assertion clearly holds {for $\alpha > \alpha_0$.}

Combining the above two cases concludes the proof. 
\end{proof}

\subsection{Proof of Theorem \ref{oracle_ineq}}

Now we are in position to prove Theorem \ref{oracle_ineq}. To obtain bounds, we will split the variance part by exploiting a technique from \citet{g10} dealing with ordered processes:
\begin{lemma}[Ordered processes~\citep{g10}]\label{od_process}
Assume that {we are given} a sequence of functions  $c_k: \mathcal{A} \to \Rset$, $k = 1, 2, \ldots$, with $\mathcal{A} \subseteq \Rset_+$, satisfying {$\sum_{k = 1}^\infty \abs{c_k(\alpha)} < \infty$ and}
\[
\kappa(\alpha) \coloneqq \left(\sum_{k = 1}^{\infty} c_k(\alpha)^2\right)^{1/2} < \infty \qquad \text{ for } \alpha \in \mathcal{A}, 
\]
and that $\kappa$ is continuous and strictly monotone on $\mathcal{A}$. Define  $(x)_+ \coloneqq \max\{x, 0\}$ and
\[
\zeta(\alpha) \coloneqq \sum_{k=1}^{\infty} c_k(\alpha) (\xi_k^2 - 1)\qquad\text{ with } \xi_k \stackrel{iid}{\sim} \mathcal{N}(0,1).
\]
Then there exists a {universal} constant $C_\xi$ such that
\[
\E{\sup_{\alpha \in \mathcal{A}}\left(\zeta(\alpha) - x\kappa(\alpha)^2\right)_+} \le \frac{C_\xi}{x} \qquad \text{ for all }x >0.
\]
\end{lemma}
\begin{proof}
Note that $\zeta(\alpha)$ is almost surely finite {by means of Chebyshev's inequality and the fact that $\Var{\zeta(\alpha)} = 2\kappa(\alpha) < \infty$ (or alternatively, by Kolmogorov's three-series theorem).} In the terminology of~\citet{g10}, $\zeta(\alpha)$ is an ordered process. The assertion can be proven in exactly the same way as Lemmata 1 and 2 by~\cite{g10}. 
\end{proof}

\begin{proof}[Proof of Theorem \ref{oracle_ineq}]
Consider the bias-variance decomposition
\begin{align}
\E{\normX{\hat{f}_{\alpha_{\mathrm{pred}}} - f}^2} & \le 2\E{\normX{(q_{\alpha_{\mathrm{pred}}}(T^*T)T^*T- I)f}^2} + 2\sigma^2\E{\normX{q_{\alpha_{\mathrm{pred}}}(T^*T)T^*\xi}^2}\nonumber \\
& = 2\E{\sum_{k=1}^{\infty} \left(1-s_{\alpha_{\mathrm{pred}}}(\lambda_k)\right)^2f_k^2} + 2\sigma^2\E{\sum_{k = 1}^{\infty} \lambda_k q_{\alpha_{\mathrm{pred}}}(\lambda_k)^2\xi_k^2} \label{eq:vb_decomp}. 
\end{align}

For the first term (i.e.~bias part) in~\eqref{eq:vb_decomp}, we have for all {$f\in\mathcal{W}_\phi\left(\rho\right)$}
\begin{align*}
\E{\sum_{k=1}^{\infty} f_k^2\left(1-s_{\alpha_{\mathrm{pred}}}(\lambda_k)\right)^2} & \le C_1 \E{\psi^{-1} \Bigl(\frac{1}{C_1}\sum_{k=1}^{\infty} \lambda_k (1-s_{\alpha_{\mathrm{pred}}}(\lambda_k))^2 f_k^2 \Bigr)} && \text{[by Lemma~\ref{source_cond}]} \\
& \le C_1 \E{\psi^{-1} \Bigl(\frac{1}{C_1}r(\alpha_{\mathrm{pred}}, f) \Bigr)} && \text{[by~\eqref{eq:pred_risk}]} \\
& \le C_1 \psi^{-1}\Bigl(\frac{1}{C_1}\E{r(\alpha_{\mathrm{pred}}, f)}\Bigr) && \text{[by Jensen's inequality]} \\
& \le C_1 \psi^{-1}\Bigl(\frac{2}{C_1} r(\alpha_o, f)+ C_2\sigma^2\Bigr), && \text{[by Remark~\ref{rmk:moments}]}
\end{align*}
{where $C_1 \coloneqq \rho^2$ with $\rho$ in~\eqref{eq:sc_general}, and $C_2$ is a universal constant.} 

For the second term (i.e.~variance part)  in~\eqref{eq:vb_decomp}, we further split it into two terms
\begin{equation}\label{eq:var}
\sigma^2\E{\sum_{k = 1}^{\infty} \lambda_k q_{\alpha_{\mathrm{pred}}}(\lambda_k)^2\xi_k^2} = \sigma^2 \E{\sum_{k = 1}^{\infty} \lambda_k q_{\alpha_{\mathrm{pred}}}(\lambda_k)^2} +  \sigma^2\E{\sum_{k = 1}^{\infty} \lambda_k q_{\alpha_{\mathrm{pred}}}(\lambda_k)^2(\xi_k^2-1)}. 
\end{equation} 
Note that $\Psi_1'(x) \le S^{-1}(x) \Psi_2'(x)$ for functions $\Psi_1$ and $\Psi_2$ in \eqref{eq:Psis}, so by monotonicity of $S$ and Assumption~\ref{surrogate_ass} (iii), condition~\eqref{eq:tail_cond} is satisfied with $\Psi(x) = \Psi_1(C_q x)$ as well. Then, for the first term in~\eqref{eq:var}, it holds that
\begin{align*}
\sigma^2 \E{\sum_{k = 1}^{\infty} \lambda_k q_{\alpha_{\mathrm{pred}}}(\lambda_k)^2} & \le C_3 \sigma^2 \E{ \Psi_1\left(C_q\sum_{k=1}^{\infty} s_{\alpha_{\mathrm{pred}}}(\lambda_k)^2\right)} &&\text{[by Lemma~\ref{compare}]} \\
& \le C_3 \sigma^2\E{\Psi_1\left(C_q\frac{r(\alpha_{\mathrm{pred}}, f)}{\sigma^2}\right)} && \text{[by~\eqref{eq:pred_risk}]} \\
& \le C_3 \sigma^{2}\left( \Psi_1\left(2C_q\frac{r(\alpha_o, f)}{\sigma^2}\right) + C_4\right) && \text{[by Corollary~\ref{ctrMoment}]} \\
& = 2 C_3 C_q \frac{r(\alpha_o, f)}{S^{-1}\left(2C_q\frac{r(\alpha_o, f)}{\sigma^2}\right)} + C_3 C_4 \sigma^2, &&
\end{align*}
{where constant $C_3$ depends only on $C_q', C_q'', C_q, C_S$, $c_q$ and the operator $T$, and constant $C_4$ is universal.}

For the second term in~\eqref{eq:var}, we apply Lemma~\ref{od_process} with $c_k(\alpha) = \lambda_k q_{\alpha}(\lambda_k)^2$, which leads to
\begin{align*}
\sigma^2\E{\sum_{k = 1}^{\infty} \lambda_k q_{\alpha_{\mathrm{pred}}}(\lambda_k)^2(\xi_k^2-1)} & \le \sigma^2 \E{\left(\zeta(\alpha_{\mathrm{pred}}) - x\kappa(\alpha_{\mathrm{pred}})^2\right)_+} +\sigma^2 \E{x\kappa(\alpha_{\mathrm{pred}})^2} \\
& \le \sigma^2 \E{\sup_{\alpha \in \mathcal{A}}\left(\zeta(\alpha) - x\kappa(\alpha)^2\right)_+} +\sigma^2 \E{x\kappa(\alpha_{\mathrm{pred}})^2} \\
& \le \sigma^2\frac{C_5}{x} + x\sigma^2\E{\sum_{k = 1}^{\infty} \lambda_k^2 q_{\alpha_{\mathrm{pred}}}(\lambda_k)^4}\qquad \text{ for all } x > 0,
\end{align*}
{where $C_5$ is a universal constant}, and $\zeta$ and $\kappa$ are defined in Lemma~\ref{od_process}.

We minimize the right hand side of the above equation over $x >0$, and then obtain 
\begin{align*}
\sigma^2\E{\sum_{k = 1}^{\infty} \lambda_k q_{\alpha_{\mathrm{pred}}}(\lambda_k)^2(\xi_k^2-1)} & \le 2\sigma^2\left(C_5 \E{\sum_{k = 1}^{\infty} \lambda_k^2 q_{\alpha_{\mathrm{pred}}}(\lambda_k)^4}\right)^{1/2} && \\
& \le C_6 \sigma^2\left(\E{\Psi_2\left(C_q \sum_{k = 1}^{\infty} s_{\alpha_{\mathrm{pred}}}(\lambda_k)^2\right)}\right)^{1/2} && \text{[by Lemma~\ref{compare}]} \\
& \le C_6 \sigma^2\left(\E{\Psi_2\left(C_q \frac{r(\alpha_{\mathrm{pred}}, f)}{\sigma^2}\right)}\right)^{1/2}  &&  \text{[by~\eqref{eq:pred_risk}]} \\
& \le C_6 \sigma^{2}\left(\Psi_2\left(2C_q\frac{r(\alpha_o, f)}{\sigma^2}\right) + C_7\right)^{1/2} && \text{[by Corollary~\ref{ctrMoment}]} \\
&\le C_6 (2C_q)^{1/2}\frac{\sigma\sqrt{r(\alpha_o, f)}}{S^{-1}\left(2 C_q \frac{r(\alpha_o, f)}{\sigma^2}\right)} + C_6(C_7)^{1/2} \sigma^2, &&
\end{align*}
{where constant $C_6$ depends only on $C_q', C_q'', C_q, C_S$, $c_q$ and the operator $T$, and constant $C_7$ is universal.}

Combining all these estimates concludes the proof.
\end{proof}

\section{Numerical simulations}\label{sec:numerics}

In this section, we will investigate the behavior of the following parameter choice methods by means of a simulation study:
\begin{enumerate}[label=(\roman*)]
\item
The oracle parameter choice $\alpha_{\mathrm{or}} = \argmin_{\alpha \in \mathcal A} \sE{\Vert \hat f_\alpha - f^\dagger\Vert_\X^2}$, which is not available in practice and evaluated here for comparison only,
\item
the a-posteriori parameter choice rule which is studied in this paper, given by  $\alpha_{\mathrm{pred}} = \argmin_{\alpha \in \mathcal A} \hat r\left(\alpha, Y\right)$ with $\hat r\left(\alpha, Y\right)$ as in \eqref{eq:risk_estimator},
\item
and the Lepski{\u\i}-type balancing principle originally introduced by \citet{l90}, and was further developed for usage in statistical inverse problems by \citet{bh05}, \citet{m06}, \citet{mp06}, and \citet{wh12}. {It consists in choosing} 
\begin{equation}\label{eq:lepskij}{
\alpha_{\mathrm{LEP}} = \max\left\{ \alpha \in \mathcal A ~\big|~ \left\Vert \hat f_{\tilde \alpha} - \hat f_{\alpha} \right\Vert_{\X} \leq 4\sigma\sqrt{\Tr\left(q_{\tilde \alpha}\left(T^*T\right)^2T^*T\right)} \text{ for all } \tilde\alpha \le \alpha, \,\tilde\alpha \in \mathcal A\right\}.}
\end{equation}
Note that the term $\sigma\sqrt{\Tr\left(q_{\alpha}\left(T^*T\right)T^*\right)}$ is in fact an estimator for the standard deviation of $\hat f_\alpha$. For an explanatory derivation of this choice we refer to \citet{m06}. Unfortunately, the computation of $\alpha_{\mathrm{LEP}}$ is expensive (see e.g. our simulations below).
\end{enumerate}

In all above methods, {for the computational purpose, we consider a discretized version of $\mathcal A$ instead by} 
\begin{equation}\label{eq:Ar}
\mathcal A_r = \left\{\sigma^2 \cdot r^k~\big|~ k=0,1,..., \left\lfloor\left(\log\left(r\right)\right)^{-1}\log\left(\sigma^{-2}\left\Vert T^*T\right\Vert\right)\right\rfloor\right\} 
\end{equation}
for some $r>1$, i.e. it discretizes the range of possible $\alpha$'s $\left[\sigma^2, \left\Vert T^*T\right\Vert\right]$ in a logarithmically equispaced way. In our simulations we use $r = 1.2$. We also tried different values of $r$ which did not influence the results significantly. Note that it can readily be seen from the error decomposition \eqref{eq:pred_risk} that the discrete parameter set $\mathcal A_r$ is -- under appropriate conditions on the filter which are satisfied by all filters in Table~\ref{tab:example} -- able to resemble the optimal behavior of a continuous parameter set $\sigma^2 \leq \alpha < \infty$ up to a constant depending on $r$.

Let us briefly comment on the implementation of the parameter choice rule $\alpha_{\mathrm{pred}}$. Even though the minimization is not performed over the continuum $\alpha \in \mathcal A$ but over a discrete set {$\mathcal A_r$} here, the computation of $\alpha_{\mathrm{pred}}$ can be numerically challenging. In Figure \ref{fig:score} we depict the function $\alpha \mapsto \hat r \left(\alpha, Y\right)$ in an example using Tikhonov regularization, which shows that the function is {relatively} flat around its absolute minimum. We observed this behavior in many situations, especially if $\sigma$ is not too small. {However, around the minimum the function is not completely flat, and the minimum seems well-defined as visible in the zooms of Figure \ref{fig:score}. This ensures that we will be able to find the minimum up to a discretization error determined by the value $r$ in \eqref{eq:Ar}.} Finally we mention that the evaluation of the trace operator in $\hat r \left(\alpha, Y\right)$ can be expensive, but this can be overcome by different techniques, cf.~\citet[Section~9.4]{ehn96} or \citet[Section~7.1]{v02}. 

\begin{figure}[!htb]
\setlength\fheight{8cm} \setlength\fwidth{12cm}
\centering
\footnotesize

\input{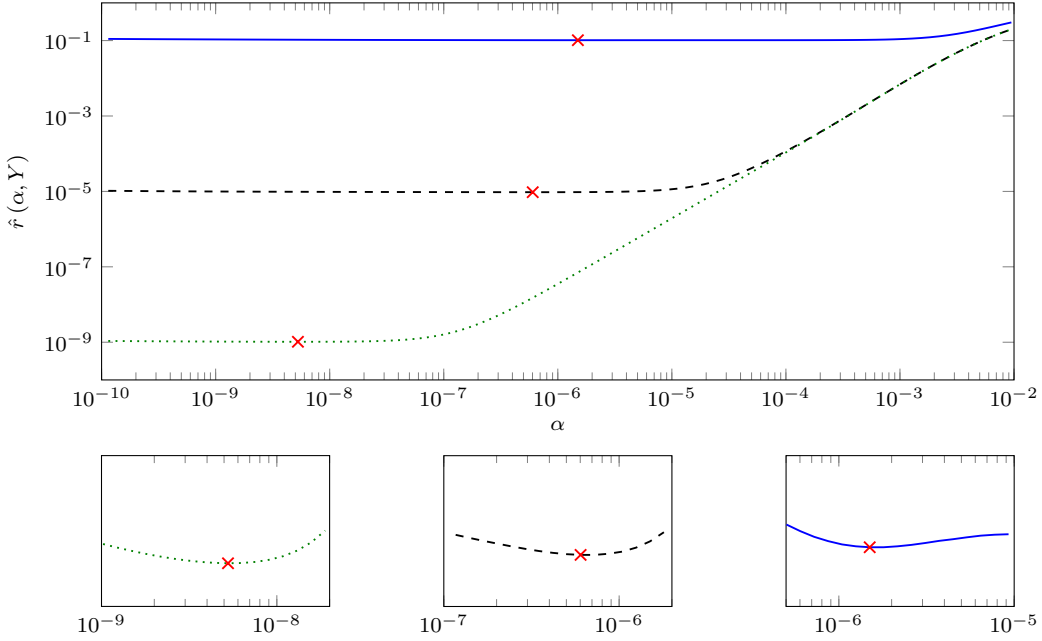}

\caption{The function $\alpha \mapsto \hat r \left(\alpha, Y\right)$ for different noise levels{: $\sigma = 10^{-2}$ (\ref{s02}), $\sigma = 10^{-4}$ (\ref{s04}), $\sigma = 10^{-6}$ (\ref{s06}). We also depict zooms (not to scale) of the regions around the actual minima, which are always marked by red crosses.} The operator $T$ is as in Section \ref{sec:numerics}.1, and $f$ is as in Example \ref{ex:1}. The chosen regularization method is Tikhonov regularization.}
\label{fig:score}
\end{figure}

\subsection{Convergence rates}

At first we investigate the empirical rate of convergence in a mildly ill-posed situation. Therefore, we consider the linear integral operator $T : \mathbf L^2 \left(\left[0,1\right]\right) \to \mathbf L^2 \left(\left[0,1\right]\right)$ defined by
\begin{equation*}
\left(Tf\right) \left(x\right) = \int_0^1 k\left(x,y\right) f\left(y\right) \,\mathrm d y, \qquad x \in \left[0,1\right]
\end{equation*}
with kernel $k\left(x,y\right) = \min\left\{x\cdot\left(1-y\right), y \cdot \left(1-x\right)\right\}, x,y \in \left[0,1\right]$, i.e. $\left(Tf\right)'' = -f$ for all $f \in \mathbf L^2 \left( \left[0,1\right] \right)$. Obviously, the eigenvalues $\lambda_k$ of $T^*T$ satisfy $\lambda_k \sim k^{-4}$. 

We discretize $T$ by choosing equidistant points $x_1 = \frac{1}{2n}, x_2 = \frac{3}{2n}, \dots, x_n = \frac{2n-1}{2n}$ and using the composite midpoint rule
\[
\left(Tf\right)\left(x\right) = \int_0^1 k\left(x,y\right) f\left(y\right) \,\mathrm d y \approx \frac{1}{n} \sum_{i=1}^n k\left(x, x_i\right) f\left(x_i\right)
\]
on the grid points $x = x_j$, $1 \leq j \leq n$. To avoid an inverse crime, the exact data $\gdag$ is always calculated analytically. {The discretization parameter $n$ is set to 1024.}

We consider two different scenarios varying in the smoothness of the unknown solution $f$:
\begin{example}\label{ex:1}
As the first example, we consider the continuous function
\[
f\left(x\right) = \begin{cases} x & \text{if }0 \leq x \leq \frac12,\\ 1-x & \text{if } \frac12 \leq x \leq 1.\end{cases}
\]
It can readily be seen by straightforward computations that the Fourier coefficients $f_k$ of $f$ are given by
\[
f_k = \frac{\left(-1\right)^k - 1}{4 \pi^3 k^2}.
\]
Consequently $f \in \mathcal{S}_{3-\varepsilon}$ and we obtain $\mathcal O \bigl(\sigma^{\frac34-\varepsilon}\bigr)$ as rate of convergence for any $\varepsilon>0$, see Corollary~\ref{order_optimal}.
\end{example}
\begin{example}\label{ex:2}
In the second example we choose
\[
f\left(x\right) = \begin{cases} 1 & \text{if }\frac14 \leq x \leq \frac34,\\ 0 & \text{else.} \end{cases}
\]
As this function can be written as the derivative of functions as in the first example, it is clear that $f \in H^{\frac12-\varepsilon} \left(\left[0,1\right]\right)$ for any $\varepsilon>0$. This is also evident by the fact that the Fourier coefficients $f_k$ of $f$ are given by
\[
f_k = \frac{\left(-1\right)^k \sin\left(\frac{\pi k}{2}\right)}{2 \pi^2 k}.
\]
Consequently, $f \in \mathcal{S}_{1-\varepsilon}$ and we obtain $\mathcal O \bigl(\sigma^{\frac13-\varepsilon}\bigr)$ as rate of convergence for any $\varepsilon>0$, see Corollary~\ref{order_optimal}.
\end{example}
\begin{remark}\label{rem:opt_rate}{
In the present setting, it follows from results by \citet{hw17} that one can obtain a rate of convergence $\mathcal O \left(\sigma^p\right)$ for a function $f$ if and only if $f \in B_{2,\infty}^{\frac{5p}{4-2p}}$ the $L^2$-based Besov-space $B_{2,\infty}^s$ with smoothness index $s$ and fine index $\infty$. Consequently, in the above mentioned Examples \ref{ex:1} and \ref{ex:2}, one finds that the minimax rates of convergence are $\mathcal O\bigl(\sigma^{\frac34}\bigr)$ and $\mathcal O \bigl(\sigma^{\frac13}\bigr)$ respectively.}
\end{remark}

In Figure \ref{fig:rate} we plot several empirical risks against the noise level $\sigma \in \left\{2^{-15}, ..., 2^{-25}\right\}$. The optimal rate of convergence {taking into account Remark \ref{rem:opt_rate}} is also indicated. We consider spectral cut-off (cf.~Figure~\ref{fig:rate}~(a)), Tikhonov regularization (cf.~Figure~\ref{fig:rate}~(b)), and Showalter regularization (cf.~Figure~\ref{fig:rate} (c)). Using Monte Carlo simulations with $10^4$ experiments per noise level we compute empirical versions of the oracle risk $R_{\mathrm{or}} \left(\sigma\right) \coloneqq \sE{\Vert \hat f_{\alpha_{\mathrm{or}}} - f^\dagger\Vert_2^2}$, the prediction risk $R_{\mathrm{pred}} \left(\sigma\right) \coloneqq \sE{\Vert \hat f_{\alpha_{\mathrm{pred}}} - f^\dagger\Vert_2^2}$, and the Lepski{\u\i} risk $R_{\mathrm{LEP}} \left(\sigma\right) \coloneqq \sE{\Vert \hat f_{\alpha_{\mathrm{LEP}}} - f^\dagger\Vert_2^2}$. 

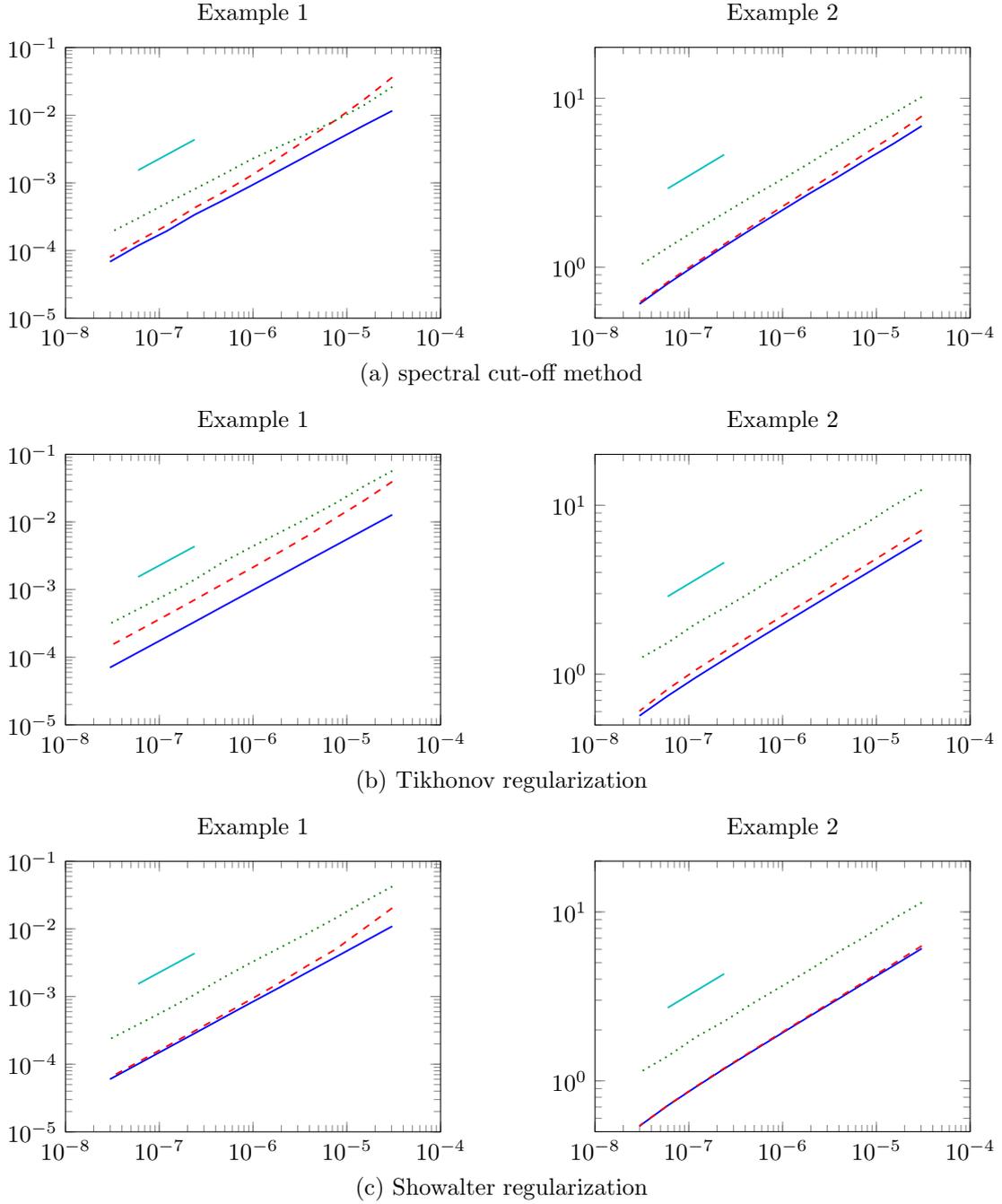
\begin{figure}[!htb]
\setlength\fheight{4cm} \setlength\fwidth{5.5cm} 
\centering
\begin{tikzpicture}[baseline]

\begin{axis}[%
width=\fwidth,
height=\fheight,
scale only axis,
xmode=log,
xmin=1e-08,
xmax=0.0001,
xminorticks=true,
ymode=log,
ymin=1e-05,
ymax=0.1,
yminorticks=true,
title={Example 1}
]
\addplot [color=blue,solid,line width = 0.25mm]
  table[row sep=crcr]{%
3.0517578125e-05	0.0116193389742968\\
1.52587890625e-05	0.00708316700922466\\
7.62939453125e-06	0.00424767729996574\\
3.814697265625e-06	0.00253857221030887\\
1.9073486328125e-06	0.00151599141912138\\
9.5367431640625e-07	0.000909108498366782\\
4.76837158203125e-07	0.000550047354042723\\
2.38418579101562e-07	0.000338714905681675\\
1.19209289550781e-07	0.000192760465580052\\
5.96046447753906e-08	0.000118430776378717\\
2.98023223876953e-08	6.82432543778183e-05\\
};
\label{or};

\addplot [color=red,dashed,line width = 0.25mm]
  table[row sep=crcr]{%
3.0517578125e-05	0.036295322375574\\
1.52587890625e-05	0.0170317344303647\\
7.62939453125e-06	0.00850433003026849\\
3.814697265625e-06	0.00449290585763362\\
1.9073486328125e-06	0.00236902889972069\\
9.5367431640625e-07	0.00128196566993519\\
4.76837158203125e-07	0.000721409217938607\\
2.38418579101562e-07	0.000429248354590549\\
1.19209289550781e-07	0.000234220033181429\\
5.96046447753906e-08	0.000138268872398077\\
2.98023223876953e-08	7.96713853312672e-05\\
};
\label{erm};

\addplot [color=black!50!green,dotted,line width = 0.25mm]
  table[row sep=crcr]{%
3.0517578125e-05	0.0262581834986209\\
1.52587890625e-05	0.0140728712809486\\
7.62939453125e-06	0.0084284562357477\\
3.814697265625e-06	0.00532025244449078\\
1.9073486328125e-06	0.00345064636306847\\
9.5367431640625e-07	0.00222583130197277\\
4.76837158203125e-07	0.00132748867264298\\
2.38418579101562e-07	0.000808437746500054\\
1.19209289550781e-07	0.000498968235594273\\
5.96046447753906e-08	0.000300098315463282\\
2.98023223876953e-08	0.000183687793735211\\
};
\label{lep};

\addplot [color=mycolor1,solid,line width = 0.25mm]
  table[row sep=crcr]{%
2.38418579101562e-07	0.00435317938454995\\
1.19209289550781e-07	0.00258841594849498\\
5.96046447753906e-08	0.00153908133126838\\
};
\label{opt};

\end{axis}
\end{tikzpicture}%
 \hspace{0.8cm} \begin{tikzpicture}[baseline]

\begin{axis}[%
width=\fwidth,
height=\fheight,
scale only axis,
xmode=log,
xmin=1e-08,
xmax=0.0001,
xminorticks=true,
ymode=log,
ymin=0.5,
ymax=20,
yminorticks=true,
title={Example 2}
]
\addplot [color=blue,solid,line width = 0.25mm]
  table[row sep=crcr]{%
3.0517578125e-05	6.85632202857599\\
1.52587890625e-05	5.38526473677365\\
7.62939453125e-06	4.29381892111309\\
3.814697265625e-06	3.39583011216897\\
1.9073486328125e-06	2.71538438714287\\
9.5367431640625e-07	2.14324932834216\\
4.76837158203125e-07	1.69139192112961\\
2.38418579101562e-07	1.32604791497141\\
1.19209289550781e-07	1.03168343986473\\
5.96046447753906e-08	0.798161546754232\\
2.98023223876953e-08	0.605165803560897\\
};

\addplot [color=red,dashed,line width = 0.25mm]
  table[row sep=crcr]{%
3.0517578125e-05	7.8297896465453\\
1.52587890625e-05	5.99505234763416\\
7.62939453125e-06	4.71018520743342\\
3.814697265625e-06	3.66335999104812\\
1.9073486328125e-06	2.87864908592692\\
9.5367431640625e-07	2.24841064974715\\
4.76837158203125e-07	1.7550433229596\\
2.38418579101562e-07	1.36497035908165\\
1.19209289550781e-07	1.06025081861093\\
5.96046447753906e-08	0.815190792716937\\
2.98023223876953e-08	0.616983323680394\\
};

\addplot [color=black!50!green,dotted,line width = 0.25mm]
  table[row sep=crcr]{%
3.0517578125e-05	10.1801781280545\\
1.52587890625e-05	8.13263806653979\\
7.62939453125e-06	6.56214084749959\\
3.814697265625e-06	5.21507103229571\\
1.9073486328125e-06	4.10781587717418\\
9.5367431640625e-07	3.27356122319601\\
4.76837158203125e-07	2.62762167705237\\
2.38418579101562e-07	2.08621742476551\\
1.19209289550781e-07	1.65878555398515\\
5.96046447753906e-08	1.3040081629473\\
2.98023223876953e-08	1.02134982133825\\
};

\addplot [color=mycolor1,solid,line width = 0.25mm]
  table[row sep=crcr]{%
2.38418579101562e-07	4.6362414217978\\
1.19209289550781e-07	3.67978725507018\\
5.96046447753906e-08	2.9206490798588\\
};

\end{axis}
\end{tikzpicture}%
 \\[0.1cm]
(a) spectral cut-off method \\[0.2cm]
\begin{tikzpicture}[baseline]

\begin{axis}[%
width=\fwidth,
height=\fheight,
scale only axis,
xmode=log,
xmin=1e-08,
xmax=0.0001,
xminorticks=true,
ymode=log,
ymin=1e-05,
ymax=0.1,
yminorticks=true,
title={Example 1}
]
\addplot [color=blue,solid,line width = 0.25mm]
  table[row sep=crcr]{%
3.0517578125e-05	0.012752223955773\\
1.52587890625e-05	0.00758614106064272\\
7.62939453125e-06	0.00450566962327946\\
3.814697265625e-06	0.00268391092295185\\
1.9073486328125e-06	0.00159066089581221\\
9.5367431640625e-07	0.000947823784740438\\
4.76837158203125e-07	0.00056351377018105\\
2.38418579101562e-07	0.000333344599470225\\
1.19209289550781e-07	0.000198271873309025\\
5.96046447753906e-08	0.000118194787219114\\
2.98023223876953e-08	7.02825563538516e-05\\
};

\addplot [color=red,dashed,line width = 0.25mm]
  table[row sep=crcr]{%
3.0517578125e-05	0.039463760507173\\
1.52587890625e-05	0.0206713712434543\\
7.62939453125e-06	0.0114634281050758\\
3.814697265625e-06	0.00621910473172134\\
1.9073486328125e-06	0.00356108653540911\\
9.5367431640625e-07	0.00205425466513262\\
4.76837158203125e-07	0.00121446374083439\\
2.38418579101562e-07	0.000704741449407461\\
1.19209289550781e-07	0.000416006122526433\\
5.96046447753906e-08	0.00024650141976843\\
2.98023223876953e-08	0.000146746381609933\\
};

\addplot [color=black!50!green,dotted,line width = 0.25mm]
  table[row sep=crcr]{%
3.0517578125e-05	0.0568557212360874\\
1.52587890625e-05	0.0337216972197594\\
7.62939453125e-06	0.0190842317386482\\
3.814697265625e-06	0.0113537945998641\\
1.9073486328125e-06	0.00693048859925879\\
9.5367431640625e-07	0.00423257527694458\\
4.76837158203125e-07	0.00254426143798829\\
2.38418579101562e-07	0.0014007814270812\\
1.19209289550781e-07	0.000841256017023017\\
5.96046447753906e-08	0.000513589291904302\\
2.98023223876953e-08	0.000313451227389835\\
};

\addplot [color=mycolor1,solid,line width = 0.25mm]
  table[row sep=crcr]{%
2.38418579101562e-07	0.00435347138383379\\
1.19209289550781e-07	0.00258858957230794\\
5.96046447753906e-08	0.00153918456860523\\
};

\end{axis}
\end{tikzpicture}%
 \hspace{0.8cm} \begin{tikzpicture}[baseline]

\begin{axis}[%
width=\fwidth,
height=\fheight,
scale only axis,
xmode=log,
xmin=1e-08,
xmax=0.0001,
xminorticks=true,
ymode=log,
ymin=0.5,
ymax=20,
yminorticks=true,
title={Example 2}
]
\addplot [color=blue,solid,line width = 0.25mm]
  table[row sep=crcr]{%
3.0517578125e-05	6.21588218391905\\
1.52587890625e-05	4.93916425961745\\
7.62939453125e-06	3.91170432508691\\
3.814697265625e-06	3.11077854510771\\
1.9073486328125e-06	2.46244573581536\\
9.5367431640625e-07	1.95105166343526\\
4.76837158203125e-07	1.54471309745981\\
2.38418579101562e-07	1.21772220877077\\
1.19209289550781e-07	0.956600936612419\\
5.96046447753906e-08	0.743465659541643\\
2.98023223876953e-08	0.566708780673558\\
};

\addplot [color=red,dashed,line width = 0.25mm]
  table[row sep=crcr]{%
3.0517578125e-05	7.11435834867334\\
1.52587890625e-05	5.58768242386655\\
7.62939453125e-06	4.39616350490106\\
3.814697265625e-06	3.47596854172589\\
1.9073486328125e-06	2.7441826853234\\
9.5367431640625e-07	2.17488163461724\\
4.76837158203125e-07	1.71729371688973\\
2.38418579101562e-07	1.35475631128682\\
1.19209289550781e-07	1.06068344894668\\
5.96046447753906e-08	0.815032210604854\\
2.98023223876953e-08	0.604213488579882\\
};

\addplot [color=black!50!green,dotted,line width = 0.25mm]
  table[row sep=crcr]{%
3.0517578125e-05	12.3483083310113\\
1.52587890625e-05	9.95133940557068\\
7.62939453125e-06	7.752363348669\\
3.814697265625e-06	6.27199786411406\\
1.9073486328125e-06	4.8710848963765\\
9.5367431640625e-07	3.93558349159953\\
4.76837158203125e-07	3.08553836076677\\
2.38418579101562e-07	2.46290336739179\\
1.19209289550781e-07	1.98934494656004\\
5.96046447753906e-08	1.53272012431252\\
2.98023223876953e-08	1.232445448962\\
};

\addplot [color=mycolor1,solid,line width = 0.25mm]
  table[row sep=crcr]{%
2.38418579101562e-07	4.57741379232898\\
1.19209289550781e-07	3.63309573461838\\
5.96046447753906e-08	2.8835899955172\\
};

\end{axis}
\end{tikzpicture}%
 \\[0.1cm]
(b) Tikhonov regularization \\[0.2cm]
\begin{tikzpicture}[baseline]

\begin{axis}[%
width=\fwidth,
height=\fheight,
scale only axis,
xmode=log,
xmin=1e-08,
xmax=0.0001,
xminorticks=true,
ymode=log,
ymin=1e-05,
ymax=0.1,
yminorticks=true,
title={Example 1}
]
\addplot [color=blue,solid,line width = 0.25mm]
  table[row sep=crcr]{%
3.0517578125e-05	0.010941761634159\\
1.52587890625e-05	0.006481896048918\\
7.62939453125e-06	0.00384410626502731\\
3.814697265625e-06	0.00229301239634098\\
1.9073486328125e-06	0.00135662076054934\\
9.5367431640625e-07	0.000812662497840448\\
4.76837158203125e-07	0.000482592534903115\\
2.38418579101562e-07	0.000285791915004106\\
1.19209289550781e-07	0.000169941103964626\\
5.96046447753906e-08	0.000100895249444331\\
2.98023223876953e-08	6.00136650888221e-05\\
};

\addplot [color=red,dashed,line width = 0.25mm]
  table[row sep=crcr]{%
3.0517578125e-05	0.0203796544183298\\
1.52587890625e-05	0.0101021492936363\\
7.62939453125e-06	0.00501217009247393\\
3.814697265625e-06	0.00287311251850359\\
1.9073486328125e-06	0.00158759659231561\\
9.5367431640625e-07	0.000911721254992627\\
4.76837158203125e-07	0.000532151091408848\\
2.38418579101562e-07	0.000309941674060695\\
1.19209289550781e-07	0.000182572049238022\\
5.96046447753906e-08	0.000107786349844247\\
2.98023223876953e-08	6.35828956200919e-05\\
};

\addplot [color=black!50!green,dotted,line width = 0.25mm]
  table[row sep=crcr]{%
3.0517578125e-05	0.0423874442555184\\
1.52587890625e-05	0.0251938195020625\\
7.62939453125e-06	0.0144795465027806\\
3.814697265625e-06	0.00859946137664919\\
1.9073486328125e-06	0.00521782494248718\\
9.5367431640625e-07	0.00317882807773154\\
4.76837158203125e-07	0.00189150226325581\\
2.38418579101562e-07	0.00107143310476553\\
1.19209289550781e-07	0.000633574379726031\\
5.96046447753906e-08	0.000386371696854129\\
2.98023223876953e-08	0.000235730838152784\\
};

\addplot [color=mycolor1,solid,line width = 0.25mm]
  table[row sep=crcr]{%
2.38418579101562e-07	0.0043531093490316\\
1.19209289550781e-07	0.00258837430512662\\
5.96046447753906e-08	0.0015390565699734\\
};

\end{axis}
\end{tikzpicture}%
 \hspace{0.8cm} \begin{tikzpicture}[baseline]

\begin{axis}[%
width=\fwidth,
height=\fheight,
scale only axis,
xmode=log,
xmin=1e-08,
xmax=0.0001,
xminorticks=true,
ymode=log,
ymin=0.5,
ymax=20,
yminorticks=true,
title={Example 2}
]
\addplot [color=blue,solid,line width = 0.25mm]
  table[row sep=crcr]{%
3.0517578125e-05	6.07076061438365\\
1.52587890625e-05	4.80392388969354\\
7.62939453125e-06	3.81846008509016\\
3.814697265625e-06	3.02254462319572\\
1.9073486328125e-06	2.39788542918086\\
9.5367431640625e-07	1.89636279871376\\
4.76837158203125e-07	1.49711731636695\\
2.38418579101562e-07	1.17903177724544\\
1.19209289550781e-07	0.92130458972445\\
5.96046447753906e-08	0.71407061105142\\
2.98023223876953e-08	0.540055662413382\\
};

\addplot [color=red,dashed,line width = 0.25mm]
  table[row sep=crcr]{%
3.0517578125e-05	6.29644960305893\\
1.52587890625e-05	4.93610214501514\\
7.62939453125e-06	3.89393366370287\\
3.814697265625e-06	3.07172313727599\\
1.9073486328125e-06	2.42672830946249\\
9.5367431640625e-07	1.9157981750536\\
4.76837158203125e-07	1.50925054028718\\
2.38418579101562e-07	1.18664052704192\\
1.19209289550781e-07	0.926135646433523\\
5.96046447753906e-08	0.715582631125816\\
2.98023223876953e-08	0.54151948746001\\
};

\addplot [color=black!50!green,dotted,line width = 0.25mm]
  table[row sep=crcr]{%
3.0517578125e-05	11.3663265339115\\
1.52587890625e-05	9.1212145117221\\
7.62939453125e-06	7.14742283241441\\
3.814697265625e-06	5.73997758533615\\
1.9073486328125e-06	4.50777823708366\\
9.5367431640625e-07	3.59968291749578\\
4.76837158203125e-07	2.8696168345186\\
2.38418579101562e-07	2.25045402666321\\
1.19209289550781e-07	1.8165557083192\\
5.96046447753906e-08	1.398812113693\\
2.98023223876953e-08	1.12024885376775\\
};

\addplot [color=mycolor1,solid,line width = 0.25mm]
  table[row sep=crcr]{%
2.38418579101562e-07	4.2992481705473\\
1.19209289550781e-07	3.41231553429957\\
5.96046447753906e-08	2.70835663439729\\
};

\end{axis}
\end{tikzpicture}%
 \\[0.1cm]
(c) Showalter regularization \\[0.2cm]
\caption{{Simulation results for various regularization methods: the oracle risk $R_{\mathrm{or}}$ (\ref{or}), the prediction risk $R_{\mathrm{prod}}$ (\ref{erm}), the Lepski{\u\i} risk $R_{\mathrm{LEP}}$ (\ref{lep}), and the optimal rates of convergence $\sigma^{3/4}$ and $\sigma^{1/3}$ respectively (\ref{opt}) as functions of $\sigma$ for the two considered examples.}}
\label{fig:rate}
\end{figure}

In all plots we find a good agreement of our theoretical predictions and the empirical results. Compared with the Lepski{\u\i}-type balancing principle, it seems that $\alpha_{\mathrm{pred}}$ performs order-optimal with a slightly smaller constant. The loss of a log-factor by using $\alpha_{\mathrm{LEP}}$ cannot be visible in such a small simulation study. {We furthermore estimated the empirical rates from the simulations depicted in Figure \ref{fig:rate} and compared them by means of statistical testing with the minimax rate of convergence. In all cases, the hypothesis test described in Appendix \ref{app:test} accepts the hypothesis that the empirical prediction risk rate is at least the minimax rate with significance level 10\%. In view of Remark \ref{rem:opt_rate}, a faster rate of convergence is impossible. From this point of view, our simulations strongly support the theory.}

\subsection{Efficiency simulations}

Besides the convergence rate simulations above we also want to numerically infer on the constant in the oracle inequality which will be done by efficiency simulations. Therefore \citep[inspired by][]{bl11,cg14} we consider the following setup. The forward operator is a $300 \times 300$ diagonal matrix with singular values $\lambda\left(k\right) = k^{-a}$ with a fixed parameter $a>0$. Then we repeat the following experiment $10^4$ times : Given a parameter $\nu$ we generate a random ground truth $f \in \mathbb R^{300}$ by $f(k)  = \pm k^{-\nu} \cdot \left(1+\mathcal N \left(0,0.1^2\right)\right)$ where the sign is independent and uniformly distributed for each component. From this ground truth, data is generated according to $Y(k) = \lambda\left(k\right) \cdot f\left(k\right) + \mathcal N \left(0,\sigma^2\right)$ where the noise is again independent in each component. Based on the data we compute empirical versions of the oracle risk $R_{\mathrm{or}} \left(\sigma\right)$, the prediction risk $R_{\mathrm{pred}} \left(\sigma\right)$, and the Lepski{\u\i} risk $R_{\mathrm{LEP}} \left(\sigma\right)$ for Tikhonov regularization. In Figure~\ref{fig:efficiencies} we depict the fractions of the oracle risk with the different a-posteriori risks for various parameters $\nu$ and $a$ to compare the average behavior of these parameter choice methods.

\begin{figure}[!htb]
\setlength\fheight{3cm} \setlength\fwidth{6cm} 
\centering
\footnotesize
\begin{tabular}{cc}
\begin{tikzpicture}[baseline]

\begin{axis}[%
width=\fwidth,
height=\fheight,
scale only axis,
xmode=log,
xmin=7e-16,
xmax=.1,
xminorticks=true,
ymode=log,
ymin=4e-5,
ymax=1.4,
yminorticks=true,
xlabel=$\sigma$,
legend pos = south west,
legend style={draw=black,fill=white,legend cell align=left}
]

\addlegendimage{empty legend}
\addlegendentry{$a = 3$, $\nu = 0.3$}

\addplot [color=black!50!green,solid,line width = 0.25mm]
  table[row sep=crcr]{%
0.1	0.00079992109991603\\
0.05	0.00126344098505539\\
0.025	0.00275924320454688\\
0.0125	0.00751773957202932\\
0.00625	0.0320755909854114\\
0.003125	0.111015251979832\\
0.0015625	0.0929908030122426\\
0.00078125	0.394052731213092\\
0.000390625	0.440805821982772\\
0.0001953125	0.646553403880778\\
9.765625e-05	0.883157404909027\\
4.8828125e-05	0.905886563472991\\
2.44140625e-05	0.953329573118967\\
1.220703125e-05	0.976049310940308\\
6.103515625e-06	0.973638446062099\\
3.0517578125e-06	0.982609556378315\\
1.52587890625e-06	0.985557114410762\\
7.62939453125e-07	0.987637559045139\\
3.814697265625e-07	0.987135752068715\\
1.9073486328125e-07	0.986991552295349\\
9.5367431640625e-08	0.986927356363361\\
4.76837158203125e-08	0.985396125639521\\
2.38418579101563e-08	0.985946688070287\\
1.19209289550781e-08	0.986608897617176\\
5.96046447753906e-09	0.990635428560936\\
2.98023223876953e-09	0.996164950435706\\
1.49011611938477e-09	0.999480350783981\\
7.45058059692383e-10	0.999988720558663\\
3.72529029846191e-10	1\\
1.86264514923096e-10	1\\
9.31322574615479e-11	0.999998181426308\\
4.65661287307739e-11	0.999993441370351\\
2.3283064365387e-11	0.999999803911719\\
1.16415321826935e-11	0.999997733020674\\
5.82076609134674e-12	0.999987327521323\\
2.91038304567337e-12	0.999998826430865\\
1.45519152283669e-12	0.999999921322386\\
7.27595761418343e-13	0.999999973640147\\
3.63797880709171e-13	0.99999856835975\\
1.81898940354586e-13	0.999999200454764\\
9.09494701772928e-14	0.999999992253525\\
4.54747350886464e-14	0.999999984164226\\
2.27373675443232e-14	0.999999991164976\\
1.13686837721616e-14	0.99999999603668\\
5.6843418860808e-15	0.999999997479481\\
2.8421709430404e-15	0.999999999832502\\
1.4210854715202e-15	0.999999999512138\\
7.105427357601e-16	0.999999999999101\\
};
\label{effURE};

\addplot [color=red,dashed,line width = 0.25mm]
  table[row sep=crcr]{%
0.1	0.973274871548626\\
0.05	0.965846450509999\\
0.025	0.957140586311694\\
0.0125	0.953670261181753\\
0.00625	0.949620734475075\\
0.003125	0.94374092122508\\
0.0015625	0.937901737556751\\
0.00078125	0.931375918558587\\
0.000390625	0.923524337967541\\
0.0001953125	0.914926193724773\\
9.765625e-05	0.905019494934478\\
4.8828125e-05	0.893418686831241\\
2.44140625e-05	0.880335520070768\\
1.220703125e-05	0.865351799921892\\
6.103515625e-06	0.84766027089686\\
3.0517578125e-06	0.825980449361132\\
1.52587890625e-06	0.801299947800063\\
7.62939453125e-07	0.769369376275098\\
3.814697265625e-07	0.731787505148443\\
1.9073486328125e-07	0.682467304477652\\
9.5367431640625e-08	0.618983085898769\\
4.76837158203125e-08	0.536906595809099\\
2.38418579101563e-08	0.42633727383498\\
1.19209289550781e-08	0.287494279106194\\
5.96046447753906e-09	0.156855124321738\\
2.98023223876953e-09	0.0772498926274657\\
1.49011611938477e-09	0.0668424361589303\\
7.45058059692383e-10	0.0646086635357977\\
3.72529029846191e-10	0.0721514507792433\\
1.86264514923096e-10	0.0756148027294791\\
9.31322574615479e-11	0.0671089827360029\\
4.65661287307739e-11	0.0649545042454498\\
2.3283064365387e-11	0.0766173356320582\\
1.16415321826935e-11	0.0708845695173777\\
5.82076609134674e-12	0.0660373637883795\\
2.91038304567337e-12	0.0623597727169227\\
1.45519152283669e-12	0.075443170387679\\
7.27595761418343e-13	0.0754451516556204\\
3.63797880709171e-13	0.0702960935197238\\
1.81898940354586e-13	0.0658759687181566\\
9.09494701772928e-14	0.0624926342017728\\
4.54747350886464e-14	0.0750094692039215\\
2.27373675443232e-14	0.0752050343733838\\
1.13686837721616e-14	0.0704143345722073\\
5.6843418860808e-15	0.0657078474912394\\
2.8421709430404e-15	0.0624343778632826\\
1.4210854715202e-15	0.0743951985863337\\
7.105427357601e-16	0.0756845586467583\\
};
\label{effLEP};

\end{axis}
\end{tikzpicture}%
 & \begin{tikzpicture}[baseline]

\begin{axis}[%
width=\fwidth,
height=\fheight,
scale only axis,
xmode=log,
xmin=7e-16,
xmax=.1,
xminorticks=true,
ymode=log,
ymin=4e-5,
ymax=1.4,
yminorticks=true,
xlabel=$\sigma$,
legend pos = south west,
legend style={draw=black,fill=white,legend cell align=left}
]

\addlegendimage{empty legend}
\addlegendentry{$a = 4$, $\nu = 0.4$}

\addplot [color=black!50!green,solid,line width = 0.25mm]
  table[row sep=crcr]{%
0.1	0.000217870694437914\\
0.05	0.000199749471058759\\
0.025	0.000215695006741389\\
0.0125	0.000332145307236776\\
0.00625	0.000348669263440446\\
0.003125	0.000554347362599421\\
0.0015625	0.000627894726265306\\
0.00078125	0.00146490232411415\\
0.000390625	0.00187711600373313\\
0.0001953125	0.00554072333153646\\
9.765625e-05	0.0254902051489822\\
4.8828125e-05	0.0307633003505766\\
2.44140625e-05	0.0312298100511064\\
1.220703125e-05	0.176446922856726\\
6.103515625e-06	0.292048435527476\\
3.0517578125e-06	0.62602675114639\\
1.52587890625e-06	0.653918121028067\\
7.62939453125e-07	0.827433667672825\\
3.814697265625e-07	0.873400842383183\\
1.9073486328125e-07	0.949925082956498\\
9.5367431640625e-08	0.960123266322171\\
4.76837158203125e-08	0.973309350298982\\
2.38418579101563e-08	0.980699057388029\\
1.19209289550781e-08	0.973452309073184\\
5.96046447753906e-09	0.98550333388483\\
2.98023223876953e-09	0.985649482661901\\
1.49011611938477e-09	0.986783680947811\\
7.45058059692383e-10	0.987120875544788\\
3.72529029846191e-10	0.985796299774388\\
1.86264514923096e-10	0.985796832653492\\
9.31322574615479e-11	0.983951821456329\\
4.65661287307739e-11	0.983002758596601\\
2.3283064365387e-11	0.984344222758541\\
1.16415321826935e-11	0.989052059029784\\
5.82076609134674e-12	0.99460150981328\\
2.91038304567337e-12	0.998489654778947\\
1.45519152283669e-12	0.999661851508232\\
7.27595761418343e-13	0.999943932392162\\
3.63797880709171e-13	0.999987394524477\\
1.81898940354586e-13	0.999998098766157\\
9.09494701772928e-14	0.99999823348261\\
4.54747350886464e-14	0.99999968248804\\
2.27373675443232e-14	0.999999867168774\\
1.13686837721616e-14	0.999993250613741\\
5.6843418860808e-15	0.999999986691467\\
2.8421709430404e-15	0.999998408359007\\
1.4210854715202e-15	0.999999944736761\\
7.105427357601e-16	0.999999863145915\\
};

\addplot [color=red,dashed,line width = 0.25mm]
  table[row sep=crcr]{%
0.1	0.969272986843605\\
0.05	0.965685951709997\\
0.025	0.951112086108039\\
0.0125	0.938931652386618\\
0.00625	0.9425946182325\\
0.003125	0.941555487973478\\
0.0015625	0.935146285707152\\
0.00078125	0.932745427194161\\
0.000390625	0.928704236915513\\
0.0001953125	0.924760872456973\\
9.765625e-05	0.920528500951787\\
4.8828125e-05	0.915888300387574\\
2.44140625e-05	0.910531689711035\\
1.220703125e-05	0.904520967981245\\
6.103515625e-06	0.898363284269454\\
3.0517578125e-06	0.891807525810973\\
1.52587890625e-06	0.884486038960839\\
7.62939453125e-07	0.876393659345058\\
3.814697265625e-07	0.866829542231259\\
1.9073486328125e-07	0.856955336855604\\
9.5367431640625e-08	0.845519035101537\\
4.76837158203125e-08	0.831685079393081\\
2.38418579101563e-08	0.816506663851713\\
1.19209289550781e-08	0.799530514210783\\
5.96046447753906e-09	0.778719445373443\\
2.98023223876953e-09	0.753737143035095\\
1.49011611938477e-09	0.724837568163652\\
7.45058059692383e-10	0.689443527126648\\
3.72529029846191e-10	0.643915586087953\\
1.86264514923096e-10	0.585290047621601\\
9.31322574615479e-11	0.510664125010529\\
4.65661287307739e-11	0.409642212693875\\
2.3283064365387e-11	0.281478822164994\\
1.16415321826935e-11	0.15791891808426\\
5.82076609134674e-12	0.0785954125226441\\
2.91038304567337e-12	0.0664119654086447\\
1.45519152283669e-12	0.0666940727978692\\
7.27595761418343e-13	0.0736683394612687\\
3.63797880709171e-13	0.0687766844348266\\
1.81898940354586e-13	0.0694594956571027\\
9.09494701772928e-14	0.0628478560376616\\
4.54747350886464e-14	0.0760700274933453\\
2.27373675443232e-14	0.0735735347263426\\
1.13686837721616e-14	0.0681285932564515\\
5.6843418860808e-15	0.0637046722751735\\
2.8421709430404e-15	0.0670866024302506\\
1.4210854715202e-15	0.0769842688334462\\
7.105427357601e-16	0.0726262729275403\\
};

\end{axis}
\end{tikzpicture}%
 \\
\begin{tikzpicture}[baseline]

\begin{axis}[%
width=\fwidth,
height=\fheight,
scale only axis,
xmode=log,
xmin=7e-16,
xmax=.1,
xminorticks=true,
ymode=log,
ymin=4e-5,
ymax=1.4,
yminorticks=true,
xlabel=$\sigma$,
legend pos = south west,
legend style={draw=black,fill=white,legend cell align=left}
]

\addlegendimage{empty legend}
\addlegendentry{$a = 5$, $\nu = 0.5$}

\addplot [color=black!50!green,solid,line width = 0.25mm]
  table[row sep=crcr]{%
0.1	8.88373561265219e-05\\
0.05	8.90835348544582e-05\\
0.025	9.52338756257208e-05\\
0.0125	0.000100374683154943\\
0.00625	9.4337494295965e-05\\
0.003125	0.000128547802727255\\
0.0015625	0.00013350967616018\\
0.00078125	0.000140366296045302\\
0.000390625	0.000220233520873779\\
0.0001953125	0.000184551714111749\\
9.765625e-05	0.000389616780669605\\
4.8828125e-05	0.000578794723639811\\
2.44140625e-05	0.000655502335256382\\
1.220703125e-05	0.00213916755842192\\
6.103515625e-06	0.00251747767364116\\
3.0517578125e-06	0.00203646921104204\\
1.52587890625e-06	0.00514014559167047\\
7.62939453125e-07	0.0644324695882992\\
3.814697265625e-07	0.0250888622223242\\
1.9073486328125e-07	0.248599509064307\\
9.5367431640625e-08	0.0777810417464837\\
4.76837158203125e-08	0.343111543934266\\
2.38418579101563e-08	0.781456356418349\\
1.19209289550781e-08	0.744937518623885\\
5.96046447753906e-09	0.85231651076684\\
2.98023223876953e-09	0.827907396006283\\
1.49011611938477e-09	0.944708806941443\\
7.45058059692383e-10	0.955779257874588\\
3.72529029846191e-10	0.963052007967872\\
1.86264514923096e-10	0.968259685476684\\
9.31322574615479e-11	0.972639188128542\\
4.65661287307739e-11	0.977051286723788\\
2.3283064365387e-11	0.983209996213952\\
1.16415321826935e-11	0.984528655775079\\
5.82076609134674e-12	0.985126530888017\\
2.91038304567337e-12	0.986914878404918\\
1.45519152283669e-12	0.986337745234152\\
7.27595761418343e-13	0.984936457100161\\
3.63797880709171e-13	0.983278328146102\\
1.81898940354586e-13	0.982141290875723\\
9.09494701772928e-14	0.980765797904367\\
4.54747350886464e-14	0.983170106076827\\
2.27373675443232e-14	0.986962500584117\\
1.13686837721616e-14	0.993030155218349\\
5.6843418860808e-15	0.997007791859657\\
2.8421709430404e-15	0.998740079408409\\
1.4210854715202e-15	0.999610030061147\\
7.105427357601e-16	0.999887806586392\\
};

\addplot [color=red,dashed,line width = 0.25mm]
  table[row sep=crcr]{%
0.1	0.961173705954702\\
0.05	0.973582044126617\\
0.025	0.955566117935094\\
0.0125	0.927428707890149\\
0.00625	0.910628746338692\\
0.003125	0.927415671630454\\
0.0015625	0.931687798200974\\
0.00078125	0.920961976176126\\
0.000390625	0.915888430851344\\
0.0001953125	0.91901988059316\\
9.765625e-05	0.9132148647673\\
4.8828125e-05	0.911705437623321\\
2.44140625e-05	0.908802806536373\\
1.220703125e-05	0.906404197437698\\
6.103515625e-06	0.902826403838907\\
3.0517578125e-06	0.899707410693763\\
1.52587890625e-06	0.895809130310337\\
7.62939453125e-07	0.892167922454243\\
3.814697265625e-07	0.888362199441155\\
1.9073486328125e-07	0.883846453194289\\
9.5367431640625e-08	0.879353241719036\\
4.76837158203125e-08	0.873763788682987\\
2.38418579101563e-08	0.868150485160963\\
1.19209289550781e-08	0.862143849768256\\
5.96046447753906e-09	0.855932387977254\\
2.98023223876953e-09	0.848334813371531\\
1.49011611938477e-09	0.840934953058661\\
7.45058059692383e-10	0.831454540530133\\
3.72529029846191e-10	0.821548066653717\\
1.86264514923096e-10	0.810576787692628\\
9.31322574615479e-11	0.798209562492302\\
4.65661287307739e-11	0.783220310666861\\
2.3283064365387e-11	0.766509854846102\\
1.16415321826935e-11	0.74803328568778\\
5.82076609134674e-12	0.724002127119253\\
2.91038304567337e-12	0.696283814009001\\
1.45519152283669e-12	0.66382764205153\\
7.27595761418343e-13	0.620743788290434\\
3.63797880709171e-13	0.564833448128765\\
1.81898940354586e-13	0.496346148478381\\
9.09494701772928e-14	0.40148953917806\\
4.54747350886464e-14	0.279801738063738\\
2.27373675443232e-14	0.161738059681073\\
1.13686837721616e-14	0.0805355295604921\\
5.6843418860808e-15	0.0659192857143492\\
2.8421709430404e-15	0.0688158684232896\\
1.4210854715202e-15	0.0701768527877993\\
7.105427357601e-16	0.064618029795434\\
};

\end{axis}
\end{tikzpicture}%
 & \begin{tikzpicture}[baseline]

\begin{axis}[%
width=\fwidth,
height=\fheight,
scale only axis,
xmode=log,
xmin=7e-16,
xmax=.1,
xminorticks=true,
ymode=log,
ymin=4e-5,
ymax=1.4,
yminorticks=true,
xlabel=$\sigma$,
legend pos = south west,
legend style={draw=black,fill=white,legend cell align=left}
]

\addlegendimage{empty legend}
\addlegendentry{$a = 6$, $\nu = 0.6$}

\addplot [color=black!50!green,solid,line width = 0.25mm]
  table[row sep=crcr]{%
0.1	4.89341855928039e-05\\
0.05	5.2755068020744e-05\\
0.025	5.93925129320518e-05\\
0.0125	5.75065973906853e-05\\
0.00625	5.2053233935156e-05\\
0.003125	4.77305627153146e-05\\
0.0015625	5.4271257577693e-05\\
0.00078125	6.00638558391999e-05\\
0.000390625	5.7863446720024e-05\\
0.0001953125	6.21001090870507e-05\\
9.765625e-05	7.67387019683043e-05\\
4.8828125e-05	7.33652512862807e-05\\
2.44140625e-05	0.00010883174664322\\
1.220703125e-05	0.000101482098411269\\
6.103515625e-06	0.000145402644051001\\
3.0517578125e-06	0.000160170256066457\\
1.52587890625e-06	0.000222623264368997\\
7.62939453125e-07	0.000261901859610707\\
3.814697265625e-07	0.000328927692921563\\
1.9073486328125e-07	0.000610519001905512\\
9.5367431640625e-08	0.000520739467941081\\
4.76837158203125e-08	0.0060247614021626\\
2.38418579101563e-08	0.0070398413543147\\
1.19209289550781e-08	0.0299770593273909\\
5.96046447753906e-09	0.011935144092274\\
2.98023223876953e-09	0.0819017777247388\\
1.49011611938477e-09	0.0761155941337671\\
7.45058059692383e-10	0.35935404200348\\
3.72529029846191e-10	0.215992518108534\\
1.86264514923096e-10	0.612014343236238\\
9.31322574615479e-11	0.606740043751235\\
4.65661287307739e-11	0.805335186872598\\
2.3283064365387e-11	0.839128905459138\\
1.16415321826935e-11	0.935205048320532\\
5.82076609134674e-12	0.930166490477391\\
2.91038304567337e-12	0.946022141694621\\
1.45519152283669e-12	0.963168232298121\\
7.27595761418343e-13	0.972533797599589\\
3.63797880709171e-13	0.971597188494818\\
1.81898940354586e-13	0.97919748103919\\
9.09494701772928e-14	0.980671106643225\\
4.54747350886464e-14	0.981671530570969\\
2.27373675443232e-14	0.984222787744149\\
1.13686837721616e-14	0.984681970870864\\
5.6843418860808e-15	0.985332641288495\\
2.8421709430404e-15	0.984667644722648\\
1.4210854715202e-15	0.98402622941952\\
7.105427357601e-16	0.982980545296578\\
};

\addplot [color=red,dashed,line width = 0.25mm]
  table[row sep=crcr]{%
0.1	0.939392150698528\\
0.05	0.977383762480174\\
0.025	0.973923120745215\\
0.0125	0.940950898588446\\
0.00625	0.893336015446509\\
0.003125	0.871260461946572\\
0.0015625	0.906565479788513\\
0.00078125	0.926166073745109\\
0.000390625	0.909643388008957\\
0.0001953125	0.893124278393445\\
9.765625e-05	0.90001212425987\\
4.8828125e-05	0.903111388008282\\
2.44140625e-05	0.893865723552338\\
1.220703125e-05	0.895345050324604\\
6.103515625e-06	0.893016765778782\\
3.0517578125e-06	0.891491380010715\\
1.52587890625e-06	0.889112783896727\\
7.62939453125e-07	0.887408511101381\\
3.814697265625e-07	0.88427009823122\\
1.9073486328125e-07	0.882841739604238\\
9.5367431640625e-08	0.880092294580603\\
4.76837158203125e-08	0.877835009793476\\
2.38418579101563e-08	0.875416254082921\\
1.19209289550781e-08	0.873040985801845\\
5.96046447753906e-09	0.869306171884325\\
2.98023223876953e-09	0.866613443264651\\
1.49011611938477e-09	0.862563296157337\\
7.45058059692383e-10	0.85892013642275\\
3.72529029846191e-10	0.854654890641425\\
1.86264514923096e-10	0.850380402889722\\
9.31322574615479e-11	0.845714428430121\\
4.65661287307739e-11	0.840303490977385\\
2.3283064365387e-11	0.834367603132577\\
1.16415321826935e-11	0.828471082501453\\
5.82076609134674e-12	0.821676893006913\\
2.91038304567337e-12	0.814145924685415\\
1.45519152283669e-12	0.805896149239956\\
7.27595761418343e-13	0.796317656349031\\
3.63797880709171e-13	0.785702169662654\\
1.81898940354586e-13	0.773815235379859\\
9.09494701772928e-14	0.759722225832337\\
4.54747350886464e-14	0.743863358099758\\
2.27373675443232e-14	0.725333433397638\\
1.13686837721616e-14	0.704393959898773\\
5.6843418860808e-15	0.676751622772579\\
2.8421709430404e-15	0.645675318949074\\
1.4210854715202e-15	0.607507497481264\\
7.105427357601e-16	0.5527966033432\\
};

\end{axis}
\end{tikzpicture}%
\end{tabular}
\caption{Efficiency simulations for Tikhonov regularization with different smoothness parameters $a$ and $\nu$: $R_{\mathrm{or}}/R_{\mathrm{pred}}$ (\ref{effURE}), $R_{\mathrm{or}}/R_{\mathrm{LEP}}$ (\ref{effLEP}).}
\label{fig:efficiencies}
\end{figure}
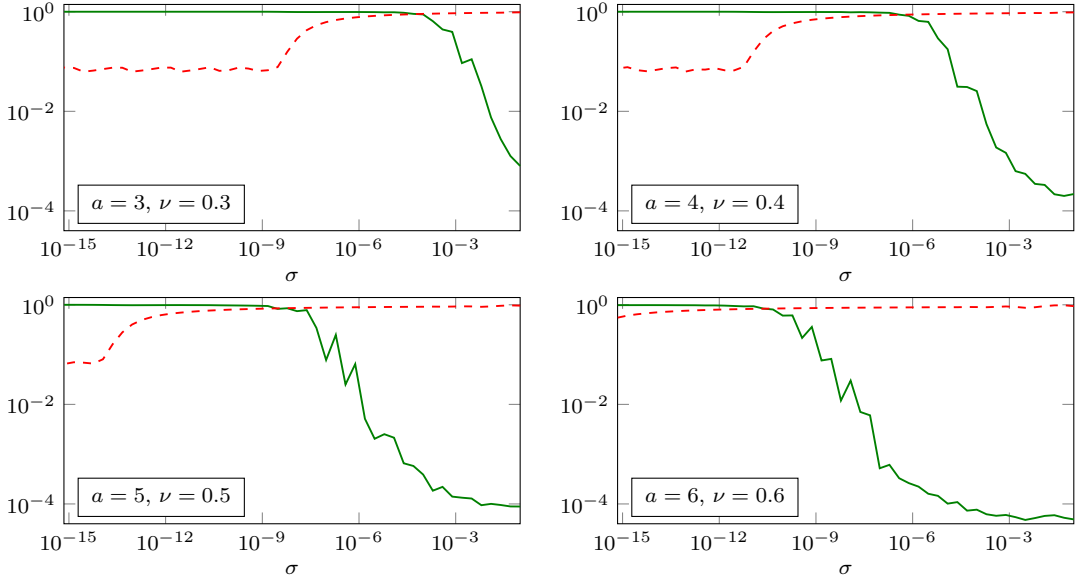

In conclusion we empirically find that both choices $\alpha_{\mathrm{pred}}$ and $\alpha_{\mathrm{LEP}}$ seem to satisfy an oracle inequality. Comparing the performance of $\alpha_{\mathrm{pred}}$ and $\alpha_{\mathrm{LEP}}$ it seems that $\alpha_{\mathrm{LEP}}$ behaves worse for small values of $\sigma$, which is in good agreement with Figure \ref{fig:rate}. Furthermore, the computational effort for $\alpha_{\mathrm{pred}}$ is significantly smaller: in our efficiency simulations around $90\%$ of the computation time were spent for computing $\alpha_{\mathrm{LEP}}$ in~\eqref{eq:lepskij}.

\section{Conclusion and outlook}\label{s:con_out}

In this study we have analyzed a parameter choice method for filter-based regularization methods applied to linear statistical inverse problems. Therefore we have proven an oracle inequality, which generalizes the one from~\citet{cg14} to general ordered filters satisfying weak assumptions (cf.~Definition \ref{def:order_filter} and Assumption~\ref{para_filter_ass}). From this oracle inequality we derived convergence rates of the investigated parameter choice, which are actually order optimal in a wide range of cases.

We point out that our techniques do not seem to be limited to the specific choice of $\alpha$ investigated here. Quite the contrary, we provide a general strategy to prove an oracle inequality, which might be used for other choices as well. If one would consider a different parameter choice rule, then an oracle inequality can be carried out the same way as in this study provided the following ingredients are available:
\begin{itemize}
\item general moment bounds for the prediction risk as in Corollary \ref{ctrMoment} (in our case based on the exponential bounds proven by~\citet{k94}),
\item a deterministic optimality result as in Lemma \ref{source_cond},
\item estimates for the behavior of the regularization algorithm as in Lemma \ref{compare}.
\end{itemize}
Note that the latter two assumptions do not rely on the parameter choice rule at all but only on the interplay of the operator, the regularization method, and the source condition. Consequently, whenever the general moments can be bounded, the analysis can basically be carried over from standard results, see e.g.~\citet{bt16} for a useful deviation inequality for regularization methods with convex penalties.

Even though we have always assume that the noise level $\sigma$ is known, the generalization to unknown $\sigma$ is straightforward. For the choice of parameter $\alpha_{\mathrm{pred}}$, we simply replace $\sigma$ by a proper estimator $\hat\sigma$ in~\eqref{eq:risk_estimator} and~\eqref{eq:alpha_pred}. The only affected part in our argumentation is the Kneip's deviation bound in Theorem~\ref{expBnd}, which still holds if we further assume the smoothness of $Tf$ (this is usually the case due to the blessing of ill-posedness). {More precisely, one could consider estimators of the form $\hat\sigma^2 \coloneqq \inner{Y}{\Lambda_\sigma Y}_{\Y^*\times\Y}$ for some linear operator $\Lambda_\sigma$ such that $\E{\inner{\xi}{\Lambda_\sigma \xi}_{\Y^*\times\Y}} = 1$ and $\Tr(\Lambda_\sigma^*\Lambda_\sigma) \le C < \infty$. In this case, under additional smoothness assumption that $\abs{\inner{Tf}{\Lambda_\sigma Tf}_{\Y}} \le \tilde C \sigma^{2}$, the assertion of Theorem~\ref{expBnd} still holds, with constants $C_\xi',C_\xi''$ there depending only on $C$ and $\tilde C$, see~\cite{gol11} for a possible choice of $\Lambda_\sigma$, and~Section 6 in \cite{k94} for further details.}

Another possible generalization concerns the errors in \eqref{eq:model}. If $\xi$ is such that the $\xi_k$'s in \eqref{eq:seq_model} are independent sub-Gaussian errors, then there are two crucial parts of the proofs which have to be generalized: Theorem~\ref{expBnd} and Lemma~\ref{od_process}. In fact it turns out that both also hold for independent sub-Gaussian errors~\citep[see][]{k94,g10}, so the whole analytical methodology remains valid in such a case as well.

The general analytical strategy advocated in this paper, of course, has its own limitations, as the resulting oracle inequality might turn out to be inadequate or even trivial in certain cases (see Section~\ref{ss:example_serious} for instance). {For exponentially ill-posed problems, we expect that the parameter choice rule under investigation has to be modified {suitably}, as examined in the seminal papers \citep{g04,cg06}}. Future questions include generalizations to nonlinear problems and noise models with heavier tails.

\section*{Acknowledgements}

FW wants to thank Yu.~Golubev for pointing his attention to the inspiring paper \citep{cg14} during a visit in G\"ottingen. FW gratefully acknowledges financial support by the German Research Foundation DFG through subproject A07 of CRC 755, and HL acknowledges support through RTG 2088, subproject B2, and the National Nature Science Foundation of China (61571008, 61402495). {We also thank two anonymous referees and the editors for several questions and constructive comments which helped us to improve the quality and presentation of the paper substantially.}

\appendix

\section{Properties of $\alpha_{\mathrm{pred}}$}\label{app:alpha}
{
We first show that almost surely the infimum of $\hat r(\alpha, Y)$ in~\eqref{eq:alpha_pred} over $\alpha$ in $\mathcal A$ is attainable, and such a minimizer is unique. The existence of minimizers follow immediately from the continuous dependence of $q_{\alpha}(\lambda)$ on $\alpha$ and the closedness of $\mathcal A$. For the uniqueness, we focus on the case that $Y_k \neq 0$ for every $k\in\mathbb N$, which holds with probability $1$. Define $\tilde r: \ell^\infty\to\Rset$ as 
$$
\tilde r(x) = \sum_{k=1}^\infty \lambda_k^2Y_k^2x_k^2 -2 \sum_{k=1}^\infty \lambda_kY_k^2x_k+2\sigma^2 \sum_{k=1}^\infty \lambda_k x_k\qquad \text{for every }x = \{x_k\}_{k\in \mathbb N}\in \ell^\infty.
$$
It is easy to see that $\tilde r(\cdot)$ is strictly convex. Note that $\hat r(\alpha, Y) = \tilde r(x)$ with $x = \{q_{\alpha}(\lambda_k)\}_{k \in\mathbb N}$. This, together with the fact that $q_{\alpha}$ is strictly increasing over $\alpha$, implies the uniqueness of $\argmin_{\alpha \in \mathcal A}\hat r(\alpha, Y)$. Thus, $\alpha_{\mathrm{pred}}$ is well-defined.}

{
Next we consider the measurability of $\alpha_{\mathrm{pred}}$. Due to its uniqueness, we have for any $x \in \Rset$
$$
\{\alpha_{\mathrm{pred}} < x\} = \Bigl\{\min_{\alpha \in\mathcal A}r(\alpha, Y) < \min_{\alpha \in\mathcal A,\, \alpha \ge x}r(\alpha, Y)\Bigr\} = \bigcup_{z \in \mathbb Q}\biggl(\Bigl\{\min_{\alpha \in\mathcal A}r(\alpha, Y) < z\Bigr\}\cap\Bigl\{\min_{\alpha \in\mathcal A,\, \alpha \ge x}r(\alpha, Y)> z\Bigr\}\biggr). 
$$
By the continuity of $q_{\alpha}$ with respect to $\alpha$, it holds that $\bigl\{\min_{\alpha \in\mathcal A}r(\alpha, Y) < z\bigr\}$ and $\bigl\{\min_{\alpha \in\mathcal A,\, \alpha \ge x}r(\alpha, Y)> z\bigr\}$ are measurable. Then $\{\alpha_{\mathrm{pred}} < x\}$ is measurable, and thus $\alpha_{\mathrm{pred}}$ is measurable. 
}

\section{Hypothesis testing for rates of convergence}\label{app:test}

{
In our simulations, we fix a test function $f$, select a sequence of noise levels $\{\sigma_i\,:\,i= 1, \ldots, n\}$, and for each $\sigma_i$ we compute estimators $\hat f_{\sigma_i,j}$ of $f$ from independent realizations of $Y$ for every $j=1,...,m$. This gives rise to an empirical estimate $\bar e_i \coloneqq \sum_{j = 1}^m e_{i,j}/m$ of the risk $\sE{\|\hat f_{\sigma_i} - f\|_\X^2}$ with $e_{i,j}\coloneqq \|\hat f_{\sigma_i, j} - f\|_\X^2$. To estimate the convergence order $\vartheta$ in $\sE{\|\hat f_{\sigma_i} - f\|_\X^2} \approx C\sigma^\vartheta$, we assume the model
\begin{equation}\label{eq:er}
\log \bar e_i \,=\, \vartheta \log \sigma_i + \varrho +\varepsilon_i\qquad \text{ with } \varepsilon_i \sim \mathcal N (0, \delta_i^2)\text{ independently, } i= 1,\ldots,n.
\end{equation}
Here $\vartheta$, and $\varrho$ are unknown, and we assume for the moment that standard deviations $\delta_i$ of perturbation are known in advance.  
Note that it is not possible to achieve faster convergence rates than the optimal one. In order to investigate the discrepancy between the convergence rate of $\hat f_\sigma$ and the optimal one, it is sufficient to test whether it is no slower than the optimal rate or not. To be precise, we consider the test
$$
H_0: \vartheta \ge \vartheta_o \qquad \text{ against }\qquad H_1: \vartheta < \vartheta_o,
$$  
where $\vartheta_o$ is the optimal order of convergence for test function $f$. }

{
From linear model theory~\citep[e.g.][]{NKNW96}, a classical testing statistics (based on the MLE estimator of $\vartheta_o$) for the above test is 
\begin{equation}\label{eq:proc}
\begin{aligned}
&\mathbb T \coloneqq 
(\hat \vartheta - \vartheta_o)\biggl(\frac{\bigl(\sum_{i = 1}^n\delta_i^{-2}\bigr)\bigl(\sum_{i = 1}^n\delta_i^{-2}\log^2\sigma_i\bigr) - \bigl(\sum_{i = 1}^n\delta_i^{-2}\log\sigma_i\bigr)^2}{\sum_{i = 1}^n\delta_i^{-2}}\biggr)^{1/2}\\ \text{ with }\quad &
\hat\vartheta = \frac{\bigl(\sum_{i=1}^n\delta_i^{-2}\bigr)\bigl(\sum_{i=1}^n\delta_i^{-2}\log \sigma_i \log \bar e_i\bigr)-\bigl(\sum_{i=1}^n\delta_i^{-2}\log \sigma_i\bigr)\bigl(\sum_{i=1}^n\delta_i^{-2} \log \bar e_i\bigr) }{\bigl(\sum_{i=1}^n\delta_i^{-2}\bigr)\bigl(\sum_{i=1}^n\delta_i^{-2}\log^2 \sigma_i\bigr)-\bigl(\sum_{i=1}^n\delta_i^{-2}\log \sigma_i\bigr)^2}.
\end{aligned}
\end{equation}
The corresponding rejection region of significance level $\alpha \in (0,1)$ is $R_\alpha\coloneqq\{\mathbb T < z_{1-\alpha}\}$, and the corresponding $p$-value is $\Phi(\mathbb T)$, where $z_{1-\alpha}$ and $\Phi$ are the ($1-\alpha$) quantile and the distribution function of the standard normal distribution, respectively.
In reality, the standard deviations $\delta_i$ are unknown, but can be easily estimated from the sample variance of $\{e_{i,j}\}_{j=1}^m$ by means of delta methods. More precisely, the central limit theorem implies $\sqrt{m}(\bar e_i - \sE{e_{i,j}})\stackrel{\mathcal D}{\to}\mathcal{N}(0, \Var{e_{i,j}})$, and via delta methods we obtain
$$
\sqrt{m}\left(\log{\bar e_i} - \log \E{e_{i,j}}\right)\stackrel{\mathcal{D}}{\to}\mathcal{N}\left(0,\, \E{e_{i,j}}^{-2}\Var{e_{i,j}}\right).
$$
Then $\Var{\log \bar e_i}\approx m^{-1}\E{e_{i,j}}^{-2}\Var{e_{i,j}}$. Based on such an approximation, we derive an estimator of $\delta_i$ as
$$
\hat\delta_i \coloneqq \frac{1}{\sqrt{m}\abs{\bar e_i}}\Bigl(\sum_{j = 1}^m (e_{i,j}- \bar e_i)^2\Bigr)^{1/2}.
$$}

{
The final procedure is given by~\eqref{eq:proc} with $\delta_i$ replaced by $\hat \delta_i$, which is exactly the testing procedure used in Section~\ref{sec:numerics}. We note that one can justify the model~\eqref{eq:er} and the estimation of $\delta_i$ by using a large $m$.}

\bibliography{copRef}

\end{document}